\documentclass[plain]{imsart}

\usepackage{natbib}
\usepackage[colorlinks,citecolor=blue,urlcolor=blue]{hyperref}

\usepackage{comment}
\usepackage[utf8]{inputenc} 
\usepackage[T1]{fontenc}    
\usepackage{hyperref}       
\usepackage{url}            
\usepackage{booktabs}       
\usepackage{amsfonts}       
\usepackage{nicefrac}       

\startlocaldefs
\usepackage{amsmath}
\usepackage{amsthm}
\usepackage{cancel}
\usepackage{dsfont}
\usepackage{graphicx}
\usepackage{mathtools}
\usepackage{xcolor}
\usepackage{enumitem}
\usepackage{appendix}

\DeclareMathOperator{\Exp}{exp}
\DeclareMathOperator{\Exps}{e}
\DeclareMathOperator{\Identity}{I}
\DeclareMathOperator{\Law}{\mathcal{L}}
\DeclareMathOperator{\Sampmed}{Med}
\DeclareMathOperator{\Theomed}{med}
\DeclareMathOperator{\Trace}{tr}
\DeclareMathOperator{\Var}{Var}

\newcommand{\abs}[1]{\left\lvert#1\right\rvert}
\newcommand{\Bandwidth}{\nu}
\newcommand{\bigo}[1]{O\left(#1\right)}
\newcommand{\card}[1]{\left\lvert#1\right\rvert}
\newcommand{\Cdf}{F}
\newcommand{\cdf}[1]{\Cdf\left(#1\right)}
\newcommand{\defeq}{\vcentcolon =}
\newcommand{\diff}{\mathop{}\mathopen{}\mathrm{d}} 
\newcommand{\Dist}{\delta}
\newcommand{\dist}[2]{\delta\left(#1,#2\right)}
\newcommand{\Ecdf}{\widehat{F}}
\renewcommand{\exp}[1]{\Exp\left(#1\right)}
\newcommand{\exps}[1]{\Exps^{#1}}
\newcommand{\Expec}{\mathbb{E}}
\newcommand{\expec}[1]{\Expec\left[#1\right]}
\newcommand{\expecunder}[2]{\Expec_{#2}\left[#1\right]}
\newcommand{\floor}[1]{\left\lfloor#1\right\rfloor}
\newcommand{\Gaussian}{\mathcal{N}}
\newcommand{\gaussian}[2]{\Gaussian\left(#1,#2\right)}
\newcommand{\Gram}{\mathrm{K}}
\newcommand{\Hilbert}{\mathcal{H}}
\newcommand{\hilbertnorm}[1]{\norm{#1}_{\Hilbert}}
\newcommand{\Indic}{\mathds{1}}
\newcommand{\indic}[1]{\Indic_{#1}}
\newcommand{\Inputspace}{\mathcal{X}}
\newcommand{\Kernel}{k}
\newcommand{\kernel}[2]{\Kernel\left(#1,#2\right)}
\newcommand{\Medvalue}{m}
\newcommand{\norm}[1]{\left\lVert#1\right\rVert}
\newcommand{\proba}[1]{\Proba\left (#1\right )}
\newcommand{\Proba}{\mathbb{P}}
\newcommand{\Reals}{\mathbb{R}}
\newcommand{\Sqmedh}{H}
\newcommand{\trace}[1]{\Trace\left(#1\right)}
\newcommand{\var}[1]{\Var\left(#1\right)}

\newcommand{\cvlaw}{\overset{\Law}{\longrightarrow}}
\newcommand{\cvproba}{\overset{\Proba}{\longrightarrow}}

\theoremstyle{plain}

\newtheorem{proposition}{Proposition}[section]
\newtheorem{lemma}{Lemma}[section]

\theoremstyle{definition}
\newtheorem{assumption}{Assumption}[section]

\theoremstyle{remark}
\newtheorem{remark}{Remark}[section]

\makeatletter
\def\hlinewd#1{%
	\noalign{\ifnum0=`}\fi\hrule \@height #1 %
	\futurelet\reserved@a\@xhline}
\makeatother

\newcommand{\MMD}{\text{MMD}}%
\newcommand{\quadMMD}{\text{MMD}_u}%
\newcommand{\linMMD}{\text{MMD}_{\ell}}%
\newcommand{\Bclass}[2]{\mathcal{D}\left(#1,#2\right)}
\newcommand{\ratioLin}{R_{\ell}}
\newcommand{\ratioQuad}{R_u}
\newcommand{\bdMedMean}{\nu_{\text{med}}^{\text{Mean}}}
\newcommand{\bdMedVar}{\nu_{\text{med}}^{\text{Var}}}
\newcommand{\bdCritMeanLin}{\nu_{\ell}^{\text{Mean}}}
\newcommand{\bdCritMeanQuad}{\nu_{u}^{\text{Mean}}}
\newcommand{\bdCritVarLin}{\nu_{\ell}^{\text{Var}}}
\newcommand{\bdCritVarQuad}{\nu_{u}^{\text{Var}}}

\setlist[itemize]{leftmargin=*}

\let\stdappendixpage\appendixpage
\renewcommand*\appendixpage{{%
   \let\markboth\relax\stdappendixpage}}

\endlocaldefs

\begin{document}

\begin{frontmatter}

\title{Large sample analysis of the median heuristic}
\runtitle{Large sample analysis of the median heuristic}


\author{\fnms{Damien} \snm{Garreau}\ead[label=e1]{damien.garreau@tuebingen.mpg.de}\thanksref{t1}\thanksref{t2}}
\thankstext{t1}{Corresponding author}
\address{\printead{e1}}
\and
\author{\fnms{Wittawat} \snm{Jitkrittum}\ead[label=e2]{wittawat.jitkrittum@tuebingen.mpg.de}\thanksref{t2}}
\affiliation{\printead{e2}}
\address{\printead{e2}}
\and
\author{\fnms{Motonobu} \snm{Kanagawa}\ead[label=e3]{motonobu.kanagawa@tuebingen.mpg.de}\thanksref{t2}}
\address{\printead{e3}}
\address{\thanksref{t2}Max Planck Institute for Intelligent Systems}
\thankstext{t2}{Max-Planck-Ring 4, 72\, 076 T\"ubingen, Germany}

\runauthor{Garreau et al.}

\begin{abstract}
In kernel methods, the median heuristic has been widely used as a way of setting the bandwidth of RBF kernels.
While its empirical performances make it a safe choice under many circumstances, there is little theoretical understanding of why this is the case.
Our aim in this paper is to advance our understanding of the median heuristic by focusing on the setting of kernel two-sample test.
We collect new findings that may be of interest for both theoreticians and practitioners.
In theory, we provide a convergence analysis that shows the asymptotic normality of the bandwidth chosen by the median heuristic in the setting of kernel two-sample test.
Systematic empirical investigations are also conducted in simple settings, comparing the performances based on the bandwidths chosen by the median heuristic and those by the maximization of test power. 
\end{abstract}



\end{frontmatter}

\section{Introduction}
\label{sec:median:introduction}

Kernel methods form an important class of algorithms in machine learning and statistics \citep{Sch_Smo:2002,MuaFukSriSch17}.
They make use of rich feature spaces that depend only on the kernel chosen by the user.
Given a positive semi-definite kernel~$\Kernel$ and observations $x_1,\ldots,x_n$, the first step of most kernel-based methods is to compute the Gram matrix $\Gram=\left(\kernel{x_i}{x_j}\right)_{1\leq i,j\leq n}$.
Thanks to the celebrated \emph{kernel trick\/}, all ensuing computations need only the knowledge of~$\Gram$.

In this paper, we are especially interested in data lying in a metric space~$\left(\Inputspace,\Dist\right)$.
When this is the case, commonly used kernels are \emph{radial basis function (RBF) kernels} of the form
\begin{equation}
\label{eq:median:rbf-kernel}
\kernel{x}{y} = f(\dist{x}{y} / \Bandwidth)
\, , \quad x,y \in \mathcal{X},
\end{equation}
where $f:\Reals_+\to\Reals_+$ is a function and $\Bandwidth$ is a positive parameter called the \emph{bandwidth}.
In many applications, the space~$\Inputspace$ is~$\Reals^d$ and~$\Dist$ is derived from the  Euclidean norm $\norm{x}=\sqrt{\sum_i x_i^2}$, that is,  $\dist{x}{y}=\norm{x-y}$.
Numerous kernels used in practice belong to this class of kernels. 
For instance, $f(x)=\exp{-x^2}$ corresponds to the Gaussian kernel~\citep{Aiz_Bra_Roz:1964}, arguably the most popular positive definite kernel used in applications \citep[see, for instance,][]{Ver_Koj_Sch:2004}.
The function $f(x)=\exp{-x}$ yields the exponential kernel---also called Laplace or Laplacian kernel---whereas more exotic~$f$ give rise to less common kernels such as the rational quadratic kernel, the wave kernel or the Mat\'ern kernel~\citep[see][and references therein]{Gen:2001}.

It is well-known that the performance of kernel methods depends highly on the kernel choice. 
In practice, this choice often reduces to the calibration of the bandwidth $\Bandwidth$, which may even be more important than the choice of~$f$ \citep[Section~4.4.5]{Sch_Smo:2002}.
Since the Gram matrix depends only on the $\norm{x_i-x_j}/\Bandwidth$ in this case, it is reasonable to pick~$\Bandwidth$ in the same order as the family of all pairwise distances $\left(\norm{x_i-x_j}\right)_{1\leq i,j\leq n}$.
As an example, suppose that we settled for the Gaussian kernel. 
Then when $\Bandwidth\to 0$, the Gram matrix~$\Gram$ is the identity matrix, and when $\Bandwidth\to\infty$, the components of $\Gram$ are all equal to~$1$.
All relevant information about the data is lost in both these extreme cases.
This is a general phenomenon, even though the values taken by~$\Gram$ in the degenerate cases depend on the function~$f$.
Hence a reasonable middle-ground for choosing~$\Bandwidth$ is to pick a value ``in the middle range'' of the $\left(\norm{x_i-x_j}\right)_{1\leq i,j\leq n}$, that is, an empirical quantile, which is often set to be the median.
This strategy is called the {\em median heuristic}; see Section~\ref{sec:def-median-heuristic} for a precise definition.

The median heuristic has been extensively used in practice.\footnote{
As noted in~\citet{Fla_Sej_Cun:2016}, the origin of the median heuristic is unclear and does not appear in the monograph of~\citet{Sch_Smo:2002}, while it has become the main reference for this heuristic.
The earliest appearance of the median heuristic that we know of is in~\citet[Section~5]{Sri_Fuk_Gre:2009}.
\citet{Gre_Bor_Ras:2012} refers to~\citet{Tak_Le_Sea:2006} and~\citet{Sch_Smo_Mul:1997} for similar heuristics.}
In fact, in unsupervised learning where no principled way is available for bandwidth selection, the median heuristic may be one of the first  choice.
This is the case for, among many others, kernel PCA \citep{SchSmoMul98}, kernel CCA \citep{BacJor02} and kernel two-sample test \citep{Gre_Bor_Ras:2012}. 
It has also been used in supervised learning, \emph{e.g.}, kernel SVM~\citep{Bos_Guy_Vap:1992} or kernel ridge regression~\citep{Hoe_Ken:1970}.
There, one can choose~$\Bandwidth$ from prescribed candidate values by performing cross-validation, and it is common to use the median heuristic to set the scale of these candidate values, \emph{e.g.}, $\Bandwidth$ may be chosen from $( 2^{a} H_n )_{a = - M}^M$, where $H_n$ denotes the value given by the median heuristic and $M$ is some positive integer.  
In fact, the empirical median is the default bandwidth choice in the kernel SVM implementation of the \texttt{kernlab} \textsf{R}  package~\citep{Kar_Smo_Hor:2004}.

Despite its popularity, there is very little theoretical understanding of the median heuristic.
To the best of our knowledge, the only work in this direction is contained in~\citet{Red_Ram_Poc:2015}.
They observe that the median of all the pairwise distances has to be close to the mean pairwise distance, $\Expec\norm{X_i-X_j}$.
Based on this observation, they obtain the asymptotic of the median heuristic when the dimension of the data goes to infinity, using the asymptotic of the mean pairwise distance.
This argument can be made rigorous by observing that, given a random variable~$X$ with a second order moment, $\abs{\Expec X - \Theomed(X)} \leq \sqrt{\var{X}}$ holds \citep{Mal:1991}.
Hence the observation of~\citet{Red_Ram_Poc:2015} is correct, up to a variance term.
We will see in Section~\ref{sec:median:main} that our results make this insight more precise.

The aim of this article is to advance our understanding of the median heuristic, both theoretically and empirically. 
To this end, we focus on the setting of the kernel two-sample test \citep{Gre_Bor_Ras:2012}, which has been used in a wide range of applications including transfer learning \citep{pmlr-v37-long15} and generative adversarial learning \citep{NIPS2017_6815}; see \citet{MuaFukSriSch17} for a recent extensive survey. 
As kernel two-sample testing is an unsupervised problem, numerous authors report the use of the median heuristic in their experiments \citep{Sri_Fuk_Gre:2009,Arl_Cel_Har:2012,Red_Ram_Poc:2015,Mua_Sri_Fuk:2016,Zha_Fil_Gre:2017,Jit_Sza_Gre:2017,Sut_Tun_Str:2017}.
On the other hand, there is a more principled approach for bandwidth (or kernel) selection based on the maximization of the test power \citep{Gre_Bor_Ras:2012,Jit_Sza_Gre:2017,Sut_Tun_Str:2017}.
Apart from theoretical analysis, we will conduct empirical comparisons between these approaches, and discuss when the median heuristic works and when it fails in simple examples; see Section~\ref{sec:example}.

This paper is organized as follows.
Our setting is made explicit in Section~\ref{sec:median:setting} and we show in the same section how it is relevant for this application.
In Section~\ref{sec:median:main}, we state our main result: the median heuristic is asymptotically normal when the number of observations goes to $+\infty$.
In particular, the median heuristic converges towards the theoretical median of a target distribution that we describe in terms of the samples' distributions.
This result is obtained by the mean of an auxiliary proposition that we think has an interest of its own, that is, a central limit theorem for a certain class of $U$-statistics that we state in the same section.
Finally, we use this result in Section~\ref{sec:example} to investigate the quality of the median heuristic as a bandwidth choice in simple settings.
While we provide sketches of proofs, the complete proofs are deferred to the Appendix. 

\section{Setting}
\label{sec:median:setting}

Given any random variable~$Z$, the notation~$Z'$ will stand for an independent copy of~$Z$ and we write $Z\sim \pi$ if~$Z$ follows the distribution of~$\pi$. 
Unless specified in subscript, the expected value is taken with respect to all the random variables that appear in the expression.
For instance, $\expec{h(X,Y)}$ stands for $\expecunder{h(X,Y)}{X,Y}$.
We also denote by $L^2(P)$ the space of real-valued measurable functions~$f$ such that $\expec{f(X)^2}<+\infty$ where $X\sim P$, with $P$ a probability distribution.
We use $\cvlaw$ for convergence in law, and $\cvproba$ for convergence in probability.

In the following, we suppose that we are given a \emph{triangular array\/} of independent $\Reals^d$-valued random variables.
Namely, for each sample size~$n$, we suppose that we observe an \emph{entirely} new sample $X_{n,1},\ldots,X_{n,n}$. 
We see later how this asymptotic setting is adapted to the methods we consider. 
Let~$X$ (resp.~$Y$) be a~$\Reals^d$-valued random variable following the law~$P$ (resp.~$Q$).
Our main assumption on the distribution of the~$X_{n,i}$ is that our observations are split in two contiguous segments
such that, on the left segment they follow~$P$, and on the right segment they follow~$Q$.
That is, 
\begin{assumption}[\textbf{Split sample}]
	\label{assump:median:distribution}
	There is a fixed $\alpha\in(0,1)$ such that, for any $i\leq \alpha n$, $X_{n,i}\sim P$ and $X_{n,i}\sim Q$ if $i>\alpha n$.
\end{assumption}
We will assume from now on that~$\alpha n$ is an integer.
Everything that follows can be readily adapted by replacing $\alpha n$ with $\floor{\alpha n}$ when it is needed.

\subsection{Connections with kernel two-sample test}
\label{sec:median:two-sample}

Let us briefly recall the modus operandi of the kernel two-sample test with our notation.  
The goal of two-sample test is to decide whether $P=Q$ or $P\neq Q$ given observations $x_1,\ldots,x_{\alpha n}$ and $x_{\alpha n+1},\ldots,x_{\alpha n}$.
\citet{Gre_Bor_Ras:2006} have proposed a kernel method for two-sample testing, that relies on the mean embeddings of~$P$ and~$Q$ in the reproducing kernel Hilbert space~$\Hilbert$ associated with~$\Kernel$.
Let us call~$\mu_P$ and~$\mu_Q$ these embeddings; then a good measure of proximity between the distributions~$P$ and~$Q$ is the so called \emph{maximum mean discrepancy\/} (MMD), which is defined as $\hilbertnorm{\mu_P-\mu_Q}$.
It is also possible to write the (squared) MMD as
\[
\MMD^2(P,Q) = \expec{\kernel{x}{x'}} - 2\expec{ \kernel{x}{y}} + \expec{\kernel{y}{y'}}
\, ,
\]
where $x,x' \sim P$ and $y,y' \sim Q$ are independent.
Setting $M=\alpha n$ and $N=(1-\alpha)n$, \citet{Gre_Bor_Ras:2006} provide an unbiased estimate of this quantity, 
\begin{align*}
\quadMMD^2(P,Q) &= \frac{1}{M(M-1)}\sum_{\substack{i,j=1\\j\neq i}}^M \kernel{x_i}{x_j} 
\\ 
&+\frac{1}{N(N-1)} \sum_{\substack{i,j=M +1\\i\neq j}}^{M+N} \kernel{x_i}{x_j} 
-\frac{2}{MN}\sum_{\substack{i=1\\j=M+1}}^{M,N}\kernel{x_i}{x_j}
\, .
\end{align*}
In the following, we refer to $\quadMMD^2$ as the \emph{quadratic-time} statistic, since its computational cost is quadratic in the number of samples. 
In the case $\alpha=1/2$, that is $M=N$, \citet{Gre_Bor_Ras:2012} also proposed a linear-time estimate, 
\begin{align*}
\linMMD^2(P,Q) &= \frac{2}{M} \sum_{i=1}^{M/2} \kernel{x_{2i-1}}{x_{2i}} + \kernel{x_{M+2i-1}}{x_{M+2i}} \\
&- \kernel{x_{2i-1}}{x_{M+2i}} - \kernel{x_{2i}}{x_{M+2i}}
\, .
\end{align*}
Letting~$n$ grow to infinity corresponds to let both~$M$ and~$N$ grow to infinity with the ratio $M/N$ constant equal to $\alpha / (1-\alpha)$.

\begin{remark}
The two-sample test problem is very close to that of \emph{change-point detection} when there is a single change-point, the main difference being that $\alpha$ is generally unknown in the change-point problem. 
Recently, some \emph{kernel} change-point detection methods have been proposed \citep{Har_Cap:2007,Arl_Cel_Har:2012}.
These methods aim to detect a change-point by minimizing a \emph{kernelized} least-squares criterion, and the median heuristic is also used in this case to set the bandwidth \citep{Arl_Cel_Har:2012,Gar_Arl:2016}. 
The setting described in this section is also appropriate to study the median heuristic in this case.
\end{remark}

\subsection{Median heuristic}
\label{sec:def-median-heuristic}

Suppose that we use a kernel that has the form~\eqref{eq:median:rbf-kernel} for a fixed~$f$ with bandwidth $\nu > 0$.
We define the median heuristic as the choice of bandwidth $\Bandwidth=\sqrt{\Sqmedh_n/2}$, where $H_n > 0$ is 
\begin{equation}
\label{eq:median:def-first-median-heuristic}
\Sqmedh_n = \Sampmed\bigl\{\norm{X_{n,i}-X_{n,j}}^2\, | \, 1\leq i < j\leq n\bigr\}
\, ,
\end{equation}
where~$\Sampmed$ is the empirical median.
More precisely, $H_n$ is obtained by first ordering the $\norm{X_{n,i}-X_{n,j}}^2$ in increasing order, and then setting $H_n$ to be the central element if $n(n-1)/2$ is odd, or the mean of the two most central elements if~$n(n-1)/2$ is even.
Note that some authors choose simply $\Bandwidth=\sqrt{\Sqmedh_n}$.

In order to investigate the asymptotic properties of~$\Sqmedh_n$, rather than using Eq.~\eqref{eq:median:def-first-median-heuristic}, we are going to define~$\Sqmedh_n$ \emph{via\/} the empirical cumulative distribution function of the $\norm{X_{n,i}-X_{n,j}}^2$, that is, 
\begin{equation}
\label{eq:median:ecdf}
\forall t\in\Reals,\quad \Ecdf_n(t) = \frac{2}{n(n-1)}\sum_{1\leq i<j\leq n} \indic{\norm{X_{n,i}-X_{n,j}}^2 \leq t}
\, .
\end{equation}
For any $p\in (0,1)$, we define the generalized inverse of~$\Ecdf_n$ by $\Ecdf^{-1}_n(p) = \inf \bigl\{t\in\Reals \, | \, \Ecdf(t)\geq p\bigr\}$.
Then $\Sqmedh_n$ may be written as
\begin{equation}
\label{eq:def:median-heuristic}
\Sqmedh_n = \Ecdf^{-1}_n\left(\frac{1}{2}\right)
\, .
\end{equation}
Note that other empirical quantiles as $\Ecdf_n^{-1}(p)$ for  $p\in(0,1)$ can be used in practice ({\em e.g.}, $p=0.1$ and $p=0.9$).
Though we are mainly concerned with~$\Sqmedh_n$, we will see that our main result still holds for arbitrary~$p$.


\section{Main results}
\label{sec:median:main}

\subsection{Empirical cumulative distribution function of the pairwise distances}
\label{sec:median:ecdf}

For any $1\leq i,j\leq n$, set $\delta_{ij}^2\defeq \norm{X_{n,i}-X_{n,j}}^2$. 
Under Assumption~\ref{assump:median:distribution}, there are only three possibilities for $\delta_{ij}^2$. 
Namely, for any fixed, distinct indices~$i$ and~$j$ such that $i<j$,   
\begin{itemize}
	\item[]
	\quad (i) if $i\leq \alpha n$ and $j \leq \alpha n$, then $\delta_{ij}^2$ has the distribution of $\norm{X-X'}^2$;
    
	\item[]
	\quad (ii) if $i>\alpha n$ and $j>\alpha n$, then $\delta_{ij}^2$ has the distribution of $\norm{Y-Y'}^2$;
    
	\item[]
	\quad (iii) if $i\leq \alpha n$ and $j>\alpha n$, then $\delta_{ij}^2$ has the  distribution of $\norm{X-Y}^2$.
\end{itemize}
In the following, we set random variables $T_{XX}\sim \norm{X-X'}^2$, $T_{YY}\sim \norm{Y-Y'}^2$, and $T_{XY}\sim \norm{X-Y}^2$. 
There are $\alpha n(\alpha n-1)/2$ occurrences of case~(i), so case~(i) occurs with proportion $\alpha^2$ as $n \to +\infty$.
Similarly, case~(ii) occurs with proportion $(1-\alpha)^2$ and case~(iii) with proportion $2\alpha(1-\alpha)$.

Define a mixture distribution~$T$ $\sim T_{XX}$, $T\sim T_{YY}$ and $T\sim T_{XY}$ with weights $\alpha^2$, $(1-\alpha)^2$ and $2\alpha(1-\alpha)$ respectively.
Thereafter, we will call~$T$ the \emph{target\/} random variable and denote by~$\Cdf$ its cumulative distribution function. 
Intuitively, when $n\to\infty$, $T_n$ should behave like a $n$-sample of the target~$T$, provided that the dependency between the $\norm{X_{n,i}-X_{n,j}}^2$ that have elements in common is not too strong. 
Indeed, a specialization of a result stated in the next paragraph shows that
\begin{equation}
\label{eq:median:ecdf-cv-simple}
\forall t\in\Reals,\qquad \Ecdf_n(t) \cvproba  \cdf{t}
\, .
\end{equation}
We believe that Eq.~\eqref{eq:median:ecdf-cv-simple} is already a step in the comprehension of the median heuristic, since we are now able to think about~$\Sqmedh_n$ ``approximately'' as the theoretical median of the target~$T$.
We refer to Figure~\ref{fig:median:ecdf-example} for two examples when $P$ and $Q$ are known distributions. 

\begin{figure}[t!]
\centering 
	\includegraphics[width=0.48\linewidth]{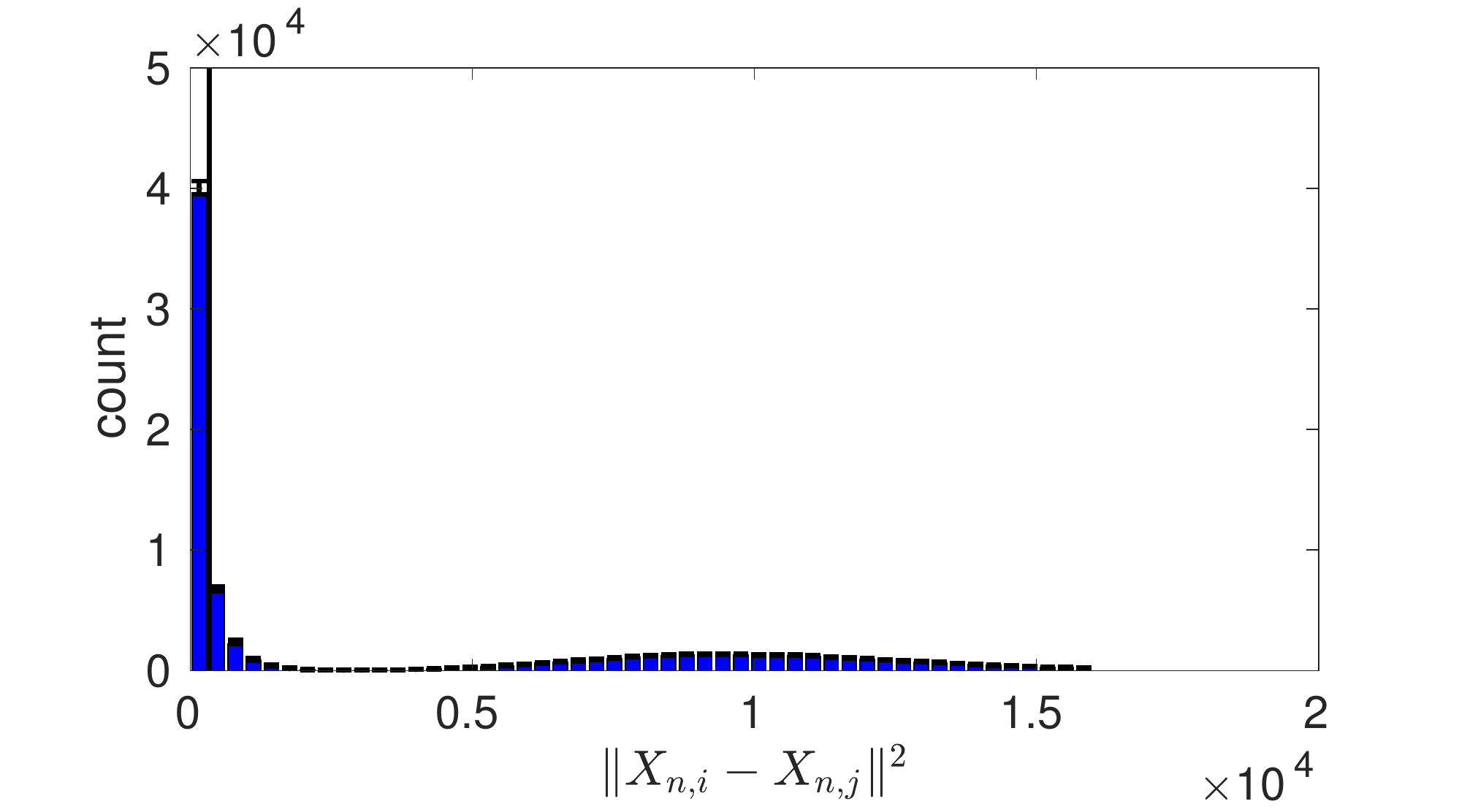}
    \hspace{-0.7cm}
	\includegraphics[width=0.49\linewidth]{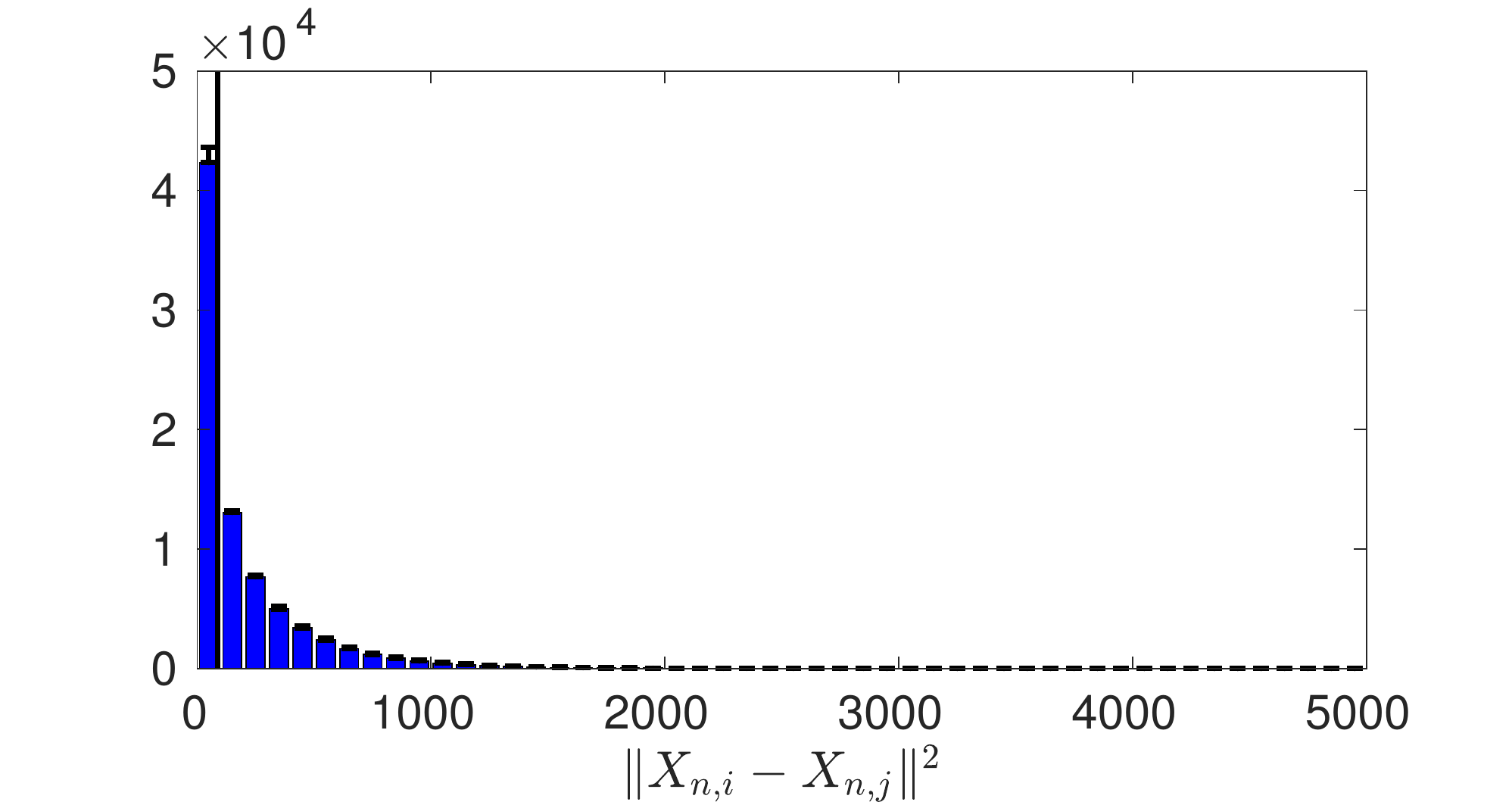}
    \vspace{-0.05in}
	\caption[Histogram of the pairwise squared-distances]{\label{fig:median:ecdf-example}Histogram of the $\norm{X_{n,i}-X_{n,j}}^2$ with $n=400$ for Gaussian distributions in dimension $d=100$ and proportion $\alpha=.25$. 
    \emph{Left panel:} change in the mean, $X\sim \gaussian{0}{\Identity_d}$ and $Y\sim\gaussian{10^3\Indic}{\Identity_d}$. 
    \emph{Right panel:} change in the variance, $X\sim \gaussian{0}{\Identity_d}$ and $Y\sim\gaussian{0}{2\Identity_d}$.
    The error bars correspond to the standard deviation over $10$ repetitions of the experiment. 
}
\end{figure}

Before presenting a rigorous statement of Eq.~\eqref{eq:median:ecdf-cv-simple}, we provide a result that shows how a gap appears between on one side $T_{XX}$ and $T_{YY}$ and $T_{XY}$ on the other side if the distributions of~$P$ and~$Q$ are well-separated.
This explains the ``two bumps'' behavior depicted in the left panel of Figure~\ref{fig:median:ecdf-example}. 
The left mode of the empirical distribution corresponds to realizations of~$T_{XX}$ and~$T_{YY}$, that are close to zero by definition, whereas the right mode corresponds to realizations of~$T_{XY}$, which can be arbitrarily far from~$0$.

\begin{lemma}[\textbf{Gap between intra- and inter-distances}]
\label{lemma:gap}
Set $\mu_X$ (resp.~$\mu_Y$) the expectation and $\Sigma_X$ (resp. $\Sigma_Y$) the covariance matrix of~$X$ (resp.~$Y$). 
Assume that there exists $\lambda > 75$ such that
\[
\norm{\mu_X-\mu_Y}^2 \geq \lambda\left(\trace{\Sigma_X} + \trace{\Sigma_Y}\right)
\, .
\]
Then, with probability at least $1-75/\lambda$, 
\[
\max\left(T_{XX},T_{YY}\right) + \frac{\norm{\mu_X-\mu_Y}^2}{25} < T_{XY}
\, .
\]
\end{lemma}

\begin{remark}
Recall that $T$ is a mixture such that $T \sim T_{XX}$, $T\sim T_{YY}$ and $T\sim T_{XY}$ with respective proportions $\alpha^2$, $(1-\alpha)^2$ and $2\alpha(1-\alpha)$.
Since $\alpha^2+(1-\alpha)^2 \geq 2\alpha (1-\alpha)$ holds, Lemma~\ref{lemma:gap} therefore implies that the median is determined by the scales of $T_{XX}$ or $T_{YY}$, and the influence of $T_{XY}$ is relatively weak if $\alpha < 1/2$. 
\end{remark}

\subsection{Convergence of the empirical cumulative distribution function}

It turns out that Eq.~\eqref{eq:median:ecdf-cv-simple} is a trivial consequence of a much stronger statement.
Indeed, $\Ecdf_n(t)$ can be seen as a sum of three dependent $U$-statistics with kernel $h(x,y)=\indic{\norm{x-y}^2\leq t}$, and the following result shows that it follows a central limit theorem.
We refer to classical textbooks~\citep{Lee:1990,Kor_Bor:2013} for an introduction to the theory of $U$-statistics.
	
\begin{proposition}[\textbf{CLT for non-identically distributed triangular array $U$-statistic}]
	\label{prop:median:clt-ustat}
	Consider $h:\Reals\times \Reals \to \Reals$ such that $h\in L^2(P)\cap L^2(Q) \times L^2(P)\cap L^2(Q)$, and suppose that the $(X_{n,i})_{i=1}^n$ satisfy Assumption~\ref{assump:median:distribution}.
	Define
	\[
	U_n = \frac{2}{n(n-1)}\sum_{1\leq i<j\leq n} h(X_{n,i},X_{n,j})
	\, ,
	\]
	and set
	\begin{equation} \label{eq:def-theta}
	\theta =\alpha^2 \expec{h(X,X')} + 2\alpha(1-\alpha)\expec{h(X,Y)} + (1-\alpha)^2\expec{h(Y,Y')}
	\, .	
	\end{equation}
	Then 
	\begin{equation}
	\label{eq:median:clt-ustat}
	\sqrt{n}(U_n-\theta) \cvlaw \gaussian{0}{\sigma^2}
	\, ,
	\end{equation}
	where $\sigma=\sigma(h,P,Q)$ is defined as
    \begin{eqnarray}
\sigma^2 
&=& {\rm Var}_{X}\left(\ \alpha \expec{h(X,X')} + (1-\alpha) \expec{h(X,Y)}\ \right) \label{eq:median:def-variance} \\
&+&  {\rm Var}_{Y} \left(\ \alpha \expec{h(X,Y)} +(1-\alpha)  \expec{h(Y,Y')} \ \right), \nonumber
\end{eqnarray}
with all of $X, X'\sim P$ and $Y, Y'\sim Q$ being independent.    
\end{proposition}
\begin{proof}[Proof (sketch)]
The idea of the proof is the following: (i) split $U_n$ in three terms depending on the relative position of the indices; (ii) write down the Hoeffding decomposition of each of these terms; (iii) show that the remainders are negligible, and (iv) conclude thanks to the central limit theorem for triangular arrays.
The complete proof can be found in Appendix \ref{sec:median:proof-clt-ustat}.
\end{proof}

We make the following remarks.
\begin{itemize}
\item
Central limit theorems for $U$-statistics are known since the fundamental article of~\citet{Hoe:1948}. 
Prop.~\ref{prop:median:clt-ustat} is in the line of such results.
An asymptotic normality result also exists in the non-identically distributed case; see \citet[Th. 8.1]{Hoe:1948}.
However, this result was not applicable in our setting.
The material in \citet{Jam_Jan:1986} covers the case of a triangular array scheme, but does not cover the non-identically distributed setting.
Results regarding \emph{two-sample\/} $U$-statistics are closest in spirit but not directly applicable; see~\citet[Section~12.2]{Vaa:1998} for an introduction and~\citet{Deh_Fri:2011} for recent developments.
With our notation, the two-sample $U$-statistic is written as $(\alpha n(1-\alpha)n)^{-1}\sum_{i=1}^{\alpha n}\sum_{j=\alpha n + 1}^n h(X_{n,i},X_{n,j})$. 
The sole difference is the absence of ``intra-segment'' interactions: the previous display does not contain terms in $h(X_{n,i},X_{n,j})$ with $i$ and~$j$ in the same segment.
It is the existence of these terms in our case which complicates the analysis.

\item
Suppose that $h$ is degenerate, that is, $ \Expec h(X,y)=\Expec h(x,Y)=0$ for $x, y \in \mathcal{X}$.
Then the variance term in Eq.~\eqref{eq:median:clt-ustat} is zero, and Proposition~\ref{prop:median:clt-ustat} remains true in the following sense: $\sqrt{n}(U_n-\theta)$ converges towards the constant $0$, which is a degenerate Gaussian distribution $\gaussian{0}{0}$.
In this case, we believe that the convergence will be faster, but not toward a Gaussian distribution.
We refer to \citet[Section~3.2.2]{Lee:1990} for results in this direction.

\item
It is possible to prove a version of Prop.~\ref{prop:median:clt-ustat} for the multiple change-point setting. 
	The proof follows the lines of Section~\ref{sec:median:proof-clt-ustat}, with an additional technical difficulty due to the numerous inter-segment interactions---we only deal with one in the present work.
\end{itemize}

\subsection{Asymptotic normality of the squared sample median}

We now turn to the statement of our main result.
In the previous section, we only obtained the convergence of the empirical distribution function.
It is well-known that such a result implies the convergence of the empirical quantiles towards the theoretical quantiles of the target distribution, if the convergence of the empirical distribution function is ``strong enough'' \citep[Chapter~21]{Vaa:1998}, {\em i.e.}, if the convergence is uniform or a CLT---as it is the case in our setting.

\begin{proposition}[\textbf{Asymptotic normality of~$\Sqmedh_n$}]
	\label{prop:median:clt-median-heuristic}
	Suppose that Assumption~\ref{assump:median:distribution} holds, and define~$T$ as in Section~\ref{sec:median:ecdf}.
	Define $\Medvalue = \Theomed(T) = F^{-1}(1/2)$ the theoretical median of the target distribution, and suppose that~$\Cdf$ has a non-zero derivative at~$\Medvalue$.
    Define $\sigma=\sigma(h,P,Q)$ as in Eq.~\eqref{eq:median:def-variance}, where $h(x,y)=\indic{\norm{x-y}^2 \leq \Medvalue}$. 
    Then, for $H_n$ defined in Eq.~\eqref{eq:def:median-heuristic}, we have
	\begin{equation}
	\label{eq:median:clt-median-heuristic}
	\sqrt{n}(\Sqmedh_n - \Medvalue) \cvlaw \gaussian{0}{\frac{\sigma^2}{\Cdf'(\Medvalue)^2}}
	\, .
	\end{equation}
\end{proposition}

Before providing a sketch of the proof of Prop.~\ref{prop:median:clt-median-heuristic}, we make a few remarks.

\begin{itemize}
\item
Empirical $U$-quantiles are known to satisfy asymptotic normality in the i.i.d.~case \citep{Ser:1980}.
Although a lot of work has been done to relax the independence assumption, however, there is no result regarding the non-identically distributed case.
In the two-sample setting, some results exist, both in the independent case~\citep{Leh:1951} and with some dependence structure~\citep{Deh_Fri:2011}.
Nevertheless, as noted before, in our setting it is necessary to consider the intra-segment interactions, which complicates the analysis.

\item
As it can be seen for instance in Figure~\ref{fig:median:ecdf-example}, observations lying between \linebreak $\max(\Expec T_{XX},\Expec T_{YY})$ and $\Expec T_{XY}$ may be very scarce.
In such a case, it is possible for~$\Cdf'(\Medvalue)$ to be small, leading to a large variance term in Eq.~\eqref{eq:median:clt-median-heuristic}.
Note however that $\Cdf'(\Medvalue)\neq 0$ for arbitrary continuous distributions.
\end{itemize}

\begin{proof}[Proof sketch]
Set $t\in\Reals$.
The general idea of the proof is to rewrite statements about the event $\bigl\{\sqrt{n} \left(\Sqmedh_n-\Medvalue\right)\leq t\bigr\}$ as statements about a sum of $U$-statistics.
We then control these $U$-statistics with Prop.~\ref{prop:median:clt-ustat} for  conveniently chosen~$h$, and conclude with Slutsky's Lemma.
Throughout this proof, we only suppose that $p\in (0,1)$ to emphasize that Prop.~\ref{prop:median:clt-median-heuristic} can be extended to \emph{any\/} quantile, not only the median.
The complete proof can be found in Appendix \ref{sec:proof-clt-median-heuristic}.
\end{proof}

\section{Empirical investigation}
\label{sec:example}

We empirically compare the median heuristic and an approach based on the maximization of test power \citep{Gre_Sej_Str:2012} in the setting of two-sample test.
The purpose is to gain insights about when the median heuristic performs reasonably and when it could work poorly.

We consider the following setting.
We focus on the use of the Gaussian kernel $k_{\Bandwidth}(x,y) \defeq \exp{\frac{-\norm{x-y}^2}{2\Bandwidth^2}}$, and consider the selection of bandwidth $\Bandwidth > 0$.
We consider the following two scenarios for the distributions~$P$ and~$Q$.
(i) {\bf Mean}: A \emph{change in the mean} of a univariate Gaussian distribution: $P\sim \gaussian{0}{1}$ and $Q\sim \gaussian{\mu}{1}$ with $\mu >0$;
(ii) {\bf Var}: A \emph{change in the variance} of a univariate Gaussian distribution: $P\sim\gaussian{0}{1}$ and $Q\sim \gaussian{0}{\sigma^2}$ with $\sigma >1$.
In both these scenarios, we choose $\alpha=1/2$ in order to be able to compute the linear-time estimate $\linMMD$. 

Since the kernel $k_\nu$ and distributions~$P$ and~$Q$ are Gaussian, it is possible to compute the cumulative distribution function $F_T$ of~$T$ exactly, and thus to have access to the ``theoretical'' value of the median heuristic given by solving in~$t$ the equation $F(t)=1/2$. 
This is done by solving this equation $F(t)=1/2$ in each scenario (see Section~\ref{sec:cdf-derivation}), which gives rise to bandwidth choices $\bdMedMean$ and $\bdMedVar$---which are functions of $\mu$ and $\sigma$, respectively. 
In this way, we suppress the noise coming from choosing empirically the median of the pairwise distance and obtain some fast, precise insights.

Comparisons are made with the bandwidth selected by the maximization of test power \citep{Gre_Sej_Str:2012}, where the criterion is defined as the (squared) MMD divided by a certain variance. 
We show in Section~\ref{sec:power-criterion} of the Appendix that, in both scenarios, it is possible to derive the power ratio criterion in closed-form, both for $\quadMMD$ and $\linMMD$. 
We denote these ratios by~$\ratioQuad$ and~$\ratioLin$ respectively. 
Let~$\bdCritMeanQuad$ and~$\bdCritVarQuad$ be the bandwidths selected the maximization by~$\ratioQuad$ in scenarios {\bf Mean} and {\bf Var} respectively, and~$\bdCritMeanLin$ and~$\bdCritVarLin$ be those by the maximization of~$\ratioLin$.
We now dispose of two different test statistics, the quadratic-time statistic $\quadMMD$ and the linear-time statistic $\linMMD$.
To each of these statistics corresponds two different way to choose the bandwidth: on one side the value given by the theoretical median, on the other side the value that maximizes the corresponding power ratio criterion, giving rise to four different tests. 

\paragraph{Comparison of bandwidths selected by different approaches.}

\begin{figure}[t!]
\centering 
\includegraphics[width=0.49\linewidth]{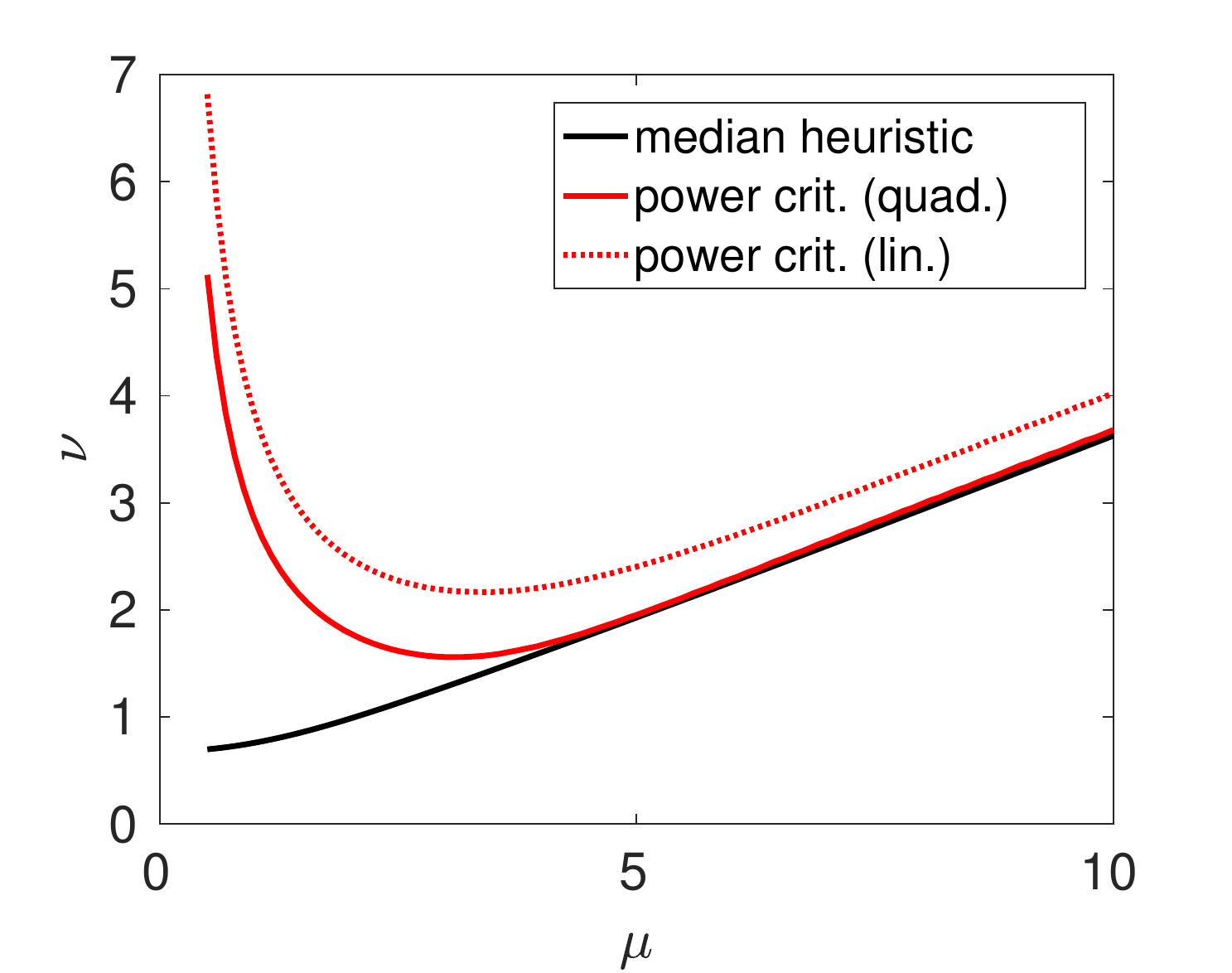}
\hspace{-0.5cm}
\includegraphics[width=0.49\linewidth]{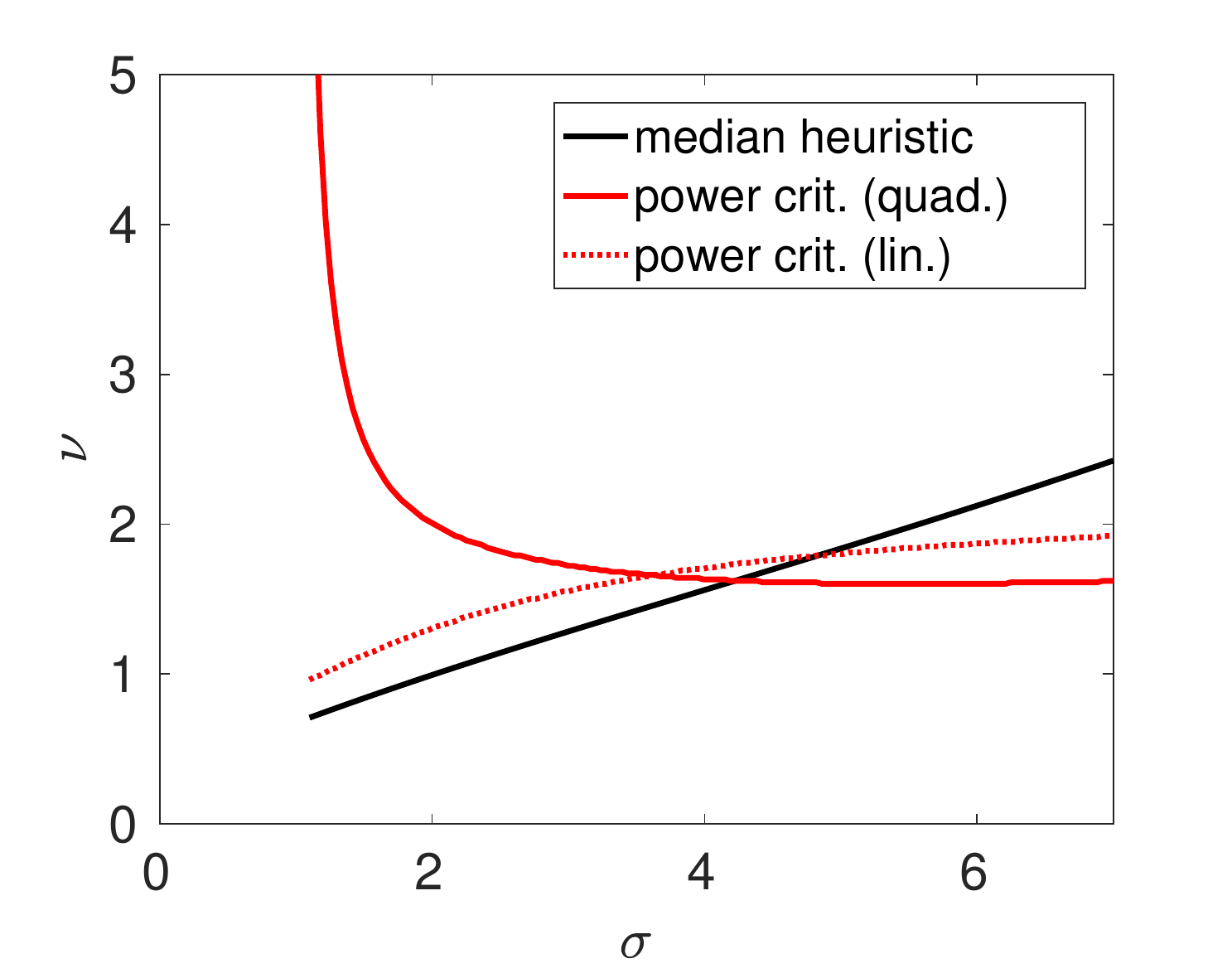}
\vspace{-0.1in}
\caption{\label{fig:comp-bd}
Gaussian kernel bandwidth selected by different means.
The \emph{left} panel corresponds to the {\bf Mean} scenario, where we vary the mean~$\mu$ of~$Q$. 
The \emph{right} panel corresponds to the {\bf Var} scenario, where we vary the variance~$\sigma^2$ of~$Q$.
For each plot, the vertical axis depicts the values of selected bandwidths. 
The black curves are obtained by computing the theoretical median, the red curves by maximizing $\ratioQuad$ (the power criterion with quadratic-time MMD), and the red dotted curves by maximizing $\ratioLin$ (the power criterion with linear-time MMD).
}
\end{figure}

Figure~\ref{fig:comp-bd} shows bandwidths selected by different approaches in the two scenarios. 
In scenario {\bf Mean}, for large values of~$\mu$, the values of~$\bdMedMean$ chosen by the median heuristic are almost identical to the values of~$\bdCritMeanQuad$ selected by the power ratio maximization with quadratic-time MMD statistics, and are parallel with $\bdCritMeanLin$ of linear-time MMD. 
We do not have a theoretical explanation for this phenomenon, which looks like a remarkable coincidence in view of the way~$\bdMedMean$ and~$\bdCritMeanQuad$ are computed---see Section~\ref{sec:cdf-derivation} and~\ref{sec:power-criterion} in the Appendix. 
On the other hand, in the {\bf Var} scenario, this behavior cannot be observed, supporting the common sense that the median heuristic may not be always the best choice, selecting a bandwidth too large for the scale of the changes.

\paragraph{Comparison of approximate Bahadur slopes.}

We turn to a measure of performance of a testing procedure, the approximate Bahadur slope \citep{Bah:1960}. 
Intuitively, this slope is the rate of convergence of $p$-values to~$0$ as~$n$ increases. 
Therefore, the higher the slope, the faster the $p$-value vanishes, and thus the lower the sample size required to reject $H_0$ under~$\theta$ (\emph{the higher the better}). 
This was first used in the context of kernel two-sample test by \citet{Jit_Xu_Sza:2017}.
We choose the ABS as a measure of quality of the two-sample testing procedures associated to the different bandwidths choices that we exposed.

Let us give a formal definition of ABS for a general testing procedure between $H_0:\theta\in \Theta_0$ and $H_1:\theta\in\Theta\setminus\Theta_0$. 
Let~$T_n$ be the associated test statistic, and denote by~$F$ the asymptotic distribution of~$T_n$ under the null. 
Assume that~$F$ is continuous and common to all $\theta_0\in\Theta_0$. 
Suppose that there exists a continuous strictly increasing function $R:\Reals_+\to\Reals_+$ such that $\lim_{n\to +\infty}R(n)=+\infty$ and $-2\log(1-F(T_n))/R(n) \cvproba c(\theta)$ where $T_n$ is drawn under $\theta$, for some function $c$ that is positive on $\Theta\setminus \Theta_0$ and identically zero on~$\Theta_0$.
Then the function~$c$ is known as the \emph{approximate Bahadur slope} (ABS) of~$T_n$. 
As it turns out, in our setting, it is possible to compute the ABS associated to~$\quadMMD$ and~$\linMMD$:

\begin{proposition}[\textbf{Approximate Bahadur Slope computations}]
\label{prop:abs-computation}
The ABS of $n\quadMMD^2$ is $\frac{\MMD^2}{4\lambda_1}$, where $\lambda_1$ is the largest eigenvalue of the centered Gaussian kernel integral operator.
The ABS of $\sqrt{n}\linMMD^2$ is $\frac{\MMD^4}{8\sigma_{\ell}^2}$, where $\sigma_{\ell}^2$ is the asymptotic variance of $\linMMD^2$ (Eq.~\eqref{eq:lin-variance} in the Appendix). 
\end{proposition}

We refer to Section~\ref{sec:proof-abs-computation} of the Appendix for the proof of Prop.~\ref{prop:abs-computation}, which is an adaptation of the proof of Th.~5 in \citet{Jit_Xu_Sza:2017}.
In the computation of the ABS for $\quadMMD^2$, we estimate~$\lambda_1$ empirically, following \citet[Eq.~(5)]{Gre_Fuk_Har:2009}.

\begin{figure}[t!]
\centering 
\includegraphics[width=0.49\linewidth]{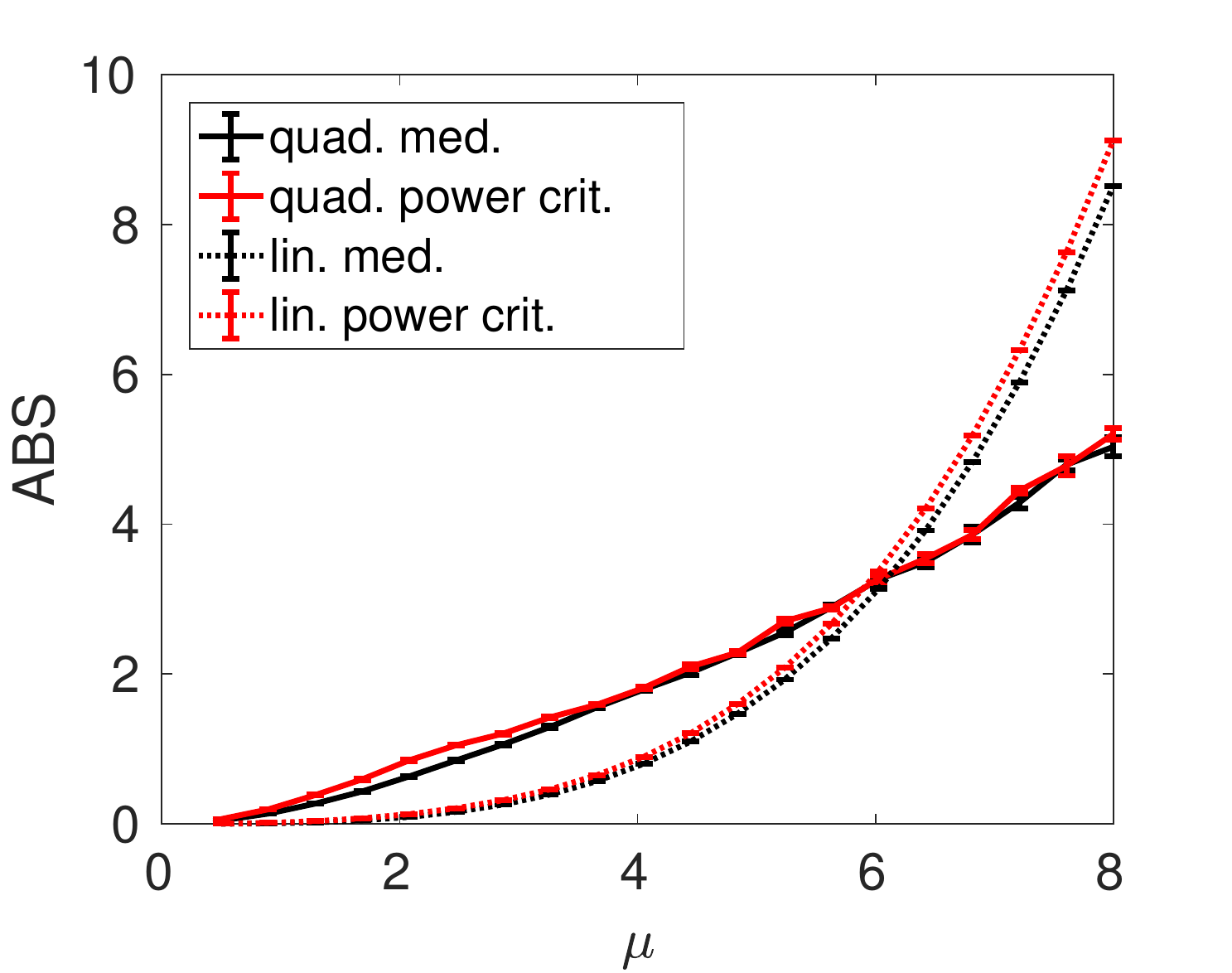}
\hspace{-0.5cm}
\includegraphics[width=0.49\linewidth]{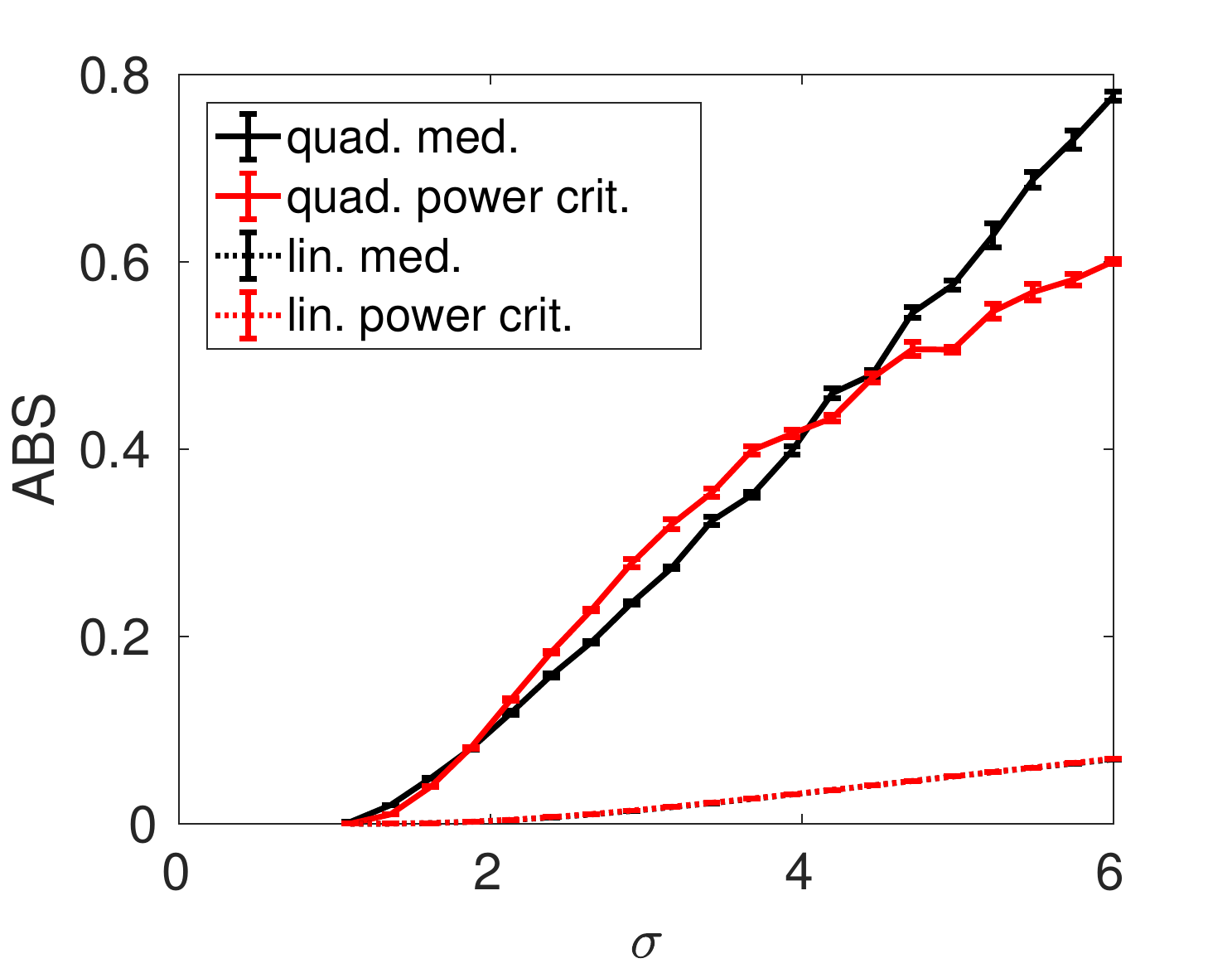}
\vspace{-0.1in}
\caption{\label{fig:abs-res}In this figure we plot the ABS of the different tests in the {\bf Mean} (\emph{left panel}) and {\bf Var} (\emph{right panel}) scenarios.
The black lines correspond to the $\quadMMD$ test with the median heuristic, the red lines to the $\quadMMD$ test with bandwidth chosen by maximizing the power criterion. 
Dotted lines correspond to the $\linMMD$ test.
We report error bars corresponding to the standard deviation over~$5$ experiments---the randomness comes from the estimation of~$\lambda_1$ from $10^3$ sample points.
}
\end{figure}

Figure~\ref{fig:abs-res} describes the approximate Bahadur slopes for the considered approaches.
It shows that, for scenario {\bf Mean} (left panel), the median heuristic provides almost identical slopes to that of the corresponding power-maximization approaches, as suggested by the results in Figure~\ref{fig:comp-bd}.
In scenario {\bf Var} (right panel), it is interesting to see that the slope by the median heuristic is almost identical to that of the power-maximization approach with linear-time MMD statistic.
A similar behavior can be observed for the quadratic time statistic, for relatively small values of the variance.

\section{Conclusion and future directions}

In this article, we partly explained the behavior of the median heuristic for large sample size, providing a convergence analysis in the setting of kernel two-sample test.
We believe that our work serves as a step towards a deeper understanding of the median heuristic.
In particular, it should open the door to more rigorous statements regarding the optimality of bandwidth choice in kernel two-sample test. 

As future work, we first aim to explain the interesting behaviors observed in Figures~\ref{fig:comp-bd} and~\ref{fig:abs-res}. 
Another direction for research is the case where~$\alpha$ depends on~$n$. 
For instance, one could be interested in a situation where we observe asymptotically much less samples of~$X$ than samples of~$Y$.
Finally, it would be extremely interesting to perform a similar analysis in machine learning settings other than kernel two-sample test. 
In particular, we suspect that for many tasks where the distributions at hand are not ``well-separated,'' the median heuristic leads to higher bandwidth than needed and is therefore a poor choice. 

\bibliographystyle{abbrvnatalt}
\bibliography{median}

\begin{thebibliography}{42}
\providecommand{\natexlab}[1]{#1}
\providecommand{\url}[1]{\texttt{#1}}
\expandafter\ifx\csname urlstyle\endcsname\relax
  \providecommand{\doi}[1]{doi: #1}\else
  \providecommand{\doi}{doi: \begingroup \urlstyle{rm}\Url}\fi

\bibitem[Abramowitz and Stegun(1964)]{Abr_Ste:1964}
Abramowitz, M. and Stegun, I. (1964).
\newblock \emph{Handbook of mathematical functions: with formulas, graphs, and
  mathematical tables}. Courier Corporation.

\bibitem[Aizerman et~al.(1964)Aizerman, Braverman, and
  Rozonoer]{Aiz_Bra_Roz:1964}
Aizerman, M., Braverman, E., and Rozonoer, L. (1964).
\newblock Theoretical foundations of the potential function method in pattern
  recognition.
\newblock \emph{Automation and Remote Control}, {\bf 25}:\penalty0 917--936.

\bibitem[Arlot et~al.(2012)Arlot, Celisse, and Harchaoui]{Arl_Cel_Har:2012}
Arlot, S., Celisse, A., and Harchaoui, Z. (2012).
\newblock A kernel multiple change-point algorithm via model selection.
\newblock ArXiV preprint, available at \url{https://arxiv.org/abs/1202.3878v2}.

\bibitem[Bach and Jordan(2002)]{BacJor02}
Bach, F. and Jordan, M.~I. (2002).
\newblock Kernel independent component analysis.
\newblock \emph{Journal of Machine Learning Research}, {\bf 3}\penalty0
  (7):\penalty0 1--48.

\bibitem[Bahadur(1960)]{Bah:1960}
Bahadur, R. (1960).
\newblock Stochastic comparison of tests.
\newblock \emph{The Annals of Mathematical Statistics}, {\bf 31}\penalty0
  (2):\penalty0 276--295.

\bibitem[Billingsley(2012)]{Bil:2012}
Billingsley, P. (2012).
\newblock \emph{Probability and Measure},   Wiley Series in Probability and
  Statistics. John Wiley \& Sons, Inc., Hoboken (NJ), third edition.

\bibitem[Boser et~al.(1992)Boser, Guyon, and Vapnik]{Bos_Guy_Vap:1992}
Boser, B.~E., Guyon, I.~M., and Vapnik, V.~N. (1992).
\newblock A training algorithm for optimal margin classifiers.
\newblock In \emph{Proceedings of the 5th annual workshop on Computational
  Learning Theory}, pp. 144--152. ACM Press.

\bibitem[Dehling and Fried(2012)]{Deh_Fri:2011}
Dehling, H. and Fried, R. (2012).
\newblock Asymptotic distribution of two-sample empirical {U}-quantiles with
  applications to robust tests for shifts in location.
\newblock \emph{Journal of Multivariate Analysis}, {\bf 105}\penalty0
  (1):\penalty0 124--140.

\bibitem[Flaxman et~al.(2016)Flaxman, Sejdinovic, Cunningham, and
  Filippi]{Fla_Sej_Cun:2016}
Flaxman, S., Sejdinovic, D., Cunningham, J.~P., and Filippi, S. (2016).
\newblock Bayesian learning of kernel embeddings.
\newblock In \emph{Proceedings of the 32th Conference on Uncertainty in
  Artificial Intelligence}, pp. 182--191.

\bibitem[Garreau and Arlot(2016)]{Gar_Arl:2016}
Garreau, D. and Arlot, S. (2016).
\newblock Consistent change-point detection with kernels.
\newblock ArXiV preprint, available at \url{http://arxiv.org/abs/1612.04740v3}.

\bibitem[Genton(2001)]{Gen:2001}
Genton, M.~G. (2001).
\newblock Classes of kernels for machine learning: a statistics perspective.
\newblock \emph{Journal of Machine Learning Research}, {\bf 2}\penalty0
  (12):\penalty0 299--312.

\bibitem[Gretton et~al.(2007)Gretton, Borgwardt, Rasch, Sch{\"o}lkopf, and
  Smola]{Gre_Bor_Ras:2006}
Gretton, A., Borgwardt, K.~M., Rasch, M., Sch{\"o}lkopf, B., and Smola, A.~J.
  (2007).
\newblock A kernel method for the two-sample-problem.
\newblock In \emph{Advances in Neural Information Processing Systems 19}, pp.
  513--520. MIT Press.

\bibitem[Gretton et~al.(2009)Gretton, Fukumizu, Harchaoui, and
  Sriperumbudur]{Gre_Fuk_Har:2009}
Gretton, A., Fukumizu, K., Harchaoui, Z., and Sriperumbudur, B. (2009).
\newblock A fast, consistent kernel two-sample test.
\newblock In \emph{Advances in Neural Information Processing Systems}, pp.
  673--681.

\bibitem[Gretton et~al.(2012{\natexlab{a}})Gretton, Borgwardt, Rasch,
  Sch{\"o}lkopf, and Smola]{Gre_Bor_Ras:2012}
Gretton, A., Borgwardt, K.~M., Rasch, M.~J., Sch{\"o}lkopf, B., and Smola, A.
  (2012{\natexlab{a}}).
\newblock A kernel two-sample test.
\newblock \emph{Journal of Machine Learning Research}, {\bf 13}\penalty0
  (1):\penalty0 723--773.

\bibitem[Gretton et~al.(2012{\natexlab{b}})Gretton, Sejdinovic, Strathmann,
  Balakrishnan, Pontil, Fukumizu, and Sriperumbudur]{Gre_Sej_Str:2012}
Gretton, A., Sejdinovic, D., Strathmann, H., Balakrishnan, S., Pontil, M.,
  Fukumizu, K., and Sriperumbudur, B.~K. (2012{\natexlab{b}}).
\newblock Optimal kernel choice for large-scale two-sample tests.
\newblock In \emph{Advances in Neural Information Processing Systems 25}, pp.
  1205--1213. Curran Associates, Inc.

\bibitem[Harchaoui and Capp{\'e}(2007)]{Har_Cap:2007}
Harchaoui, Z. and Capp{\'e}, O. (2007).
\newblock Retrospective multiple change-point estimation with kernels.
\newblock In \emph{Proceedings of the 14th IEEE Workshop on Statistical Signal
  Processing}, pp. 768--772.

\bibitem[Hoeffding(1948)]{Hoe:1948}
Hoeffding, W. (1948).
\newblock A class of statistics with asymptotically normal distribution.
\newblock \emph{The Annals of Mathematical Statistics}, {\bf 19}\penalty0
  (3):\penalty0 293--325.

\bibitem[Hoerl and Kennard(1970)]{Hoe_Ken:1970}
Hoerl, A.~E. and Kennard, R.~W. (1970).
\newblock Ridge regression: Biased estimation for nonorthogonal problems.
\newblock \emph{Technometrics}, {\bf 12}\penalty0 (1):\penalty0 55--67.

\bibitem[Jammalamadaka and Janson(1986)]{Jam_Jan:1986}
Jammalamadaka, S.~R. and Janson, S. (1986).
\newblock Limit theorems for a triangular scheme of {U}-statistics with
  applications to inter-point distances.
\newblock \emph{The Annals of Probability}, pp. 1347--1358.

\bibitem[Jitkrittum et~al.(2017{\natexlab{a}})Jitkrittum, Szab{\'o}, and
  Gretton]{Jit_Sza_Gre:2017}
Jitkrittum, W., Szab{\'o}, Z., and Gretton, A. (2017{\natexlab{a}}).
\newblock An adaptive test of independence with analytic kernel embeddings.
\newblock In \emph{Proceedings of the 34th International Conference on Machine
  Learning}.

\bibitem[Jitkrittum et~al.(2017{\natexlab{b}})Jitkrittum, Xu, Szab{\'o},
  Fukumizu, and Gretton]{Jit_Xu_Sza:2017}
Jitkrittum, W., Xu, W., Szab{\'o}, Z., Fukumizu, K., and Gretton, A.
  (2017{\natexlab{b}}).
\newblock A linear-time kernel goodness-of-fit test.
\newblock In \emph{Advances in Neural Information Processing Systems}, pp.
  261--270.

\bibitem[Karatzoglou et~al.(2004)Karatzoglou, Smola, Hornik, and
  Zeileis]{Kar_Smo_Hor:2004}
Karatzoglou, A., Smola, A., Hornik, K., and Zeileis, A. (2004).
\newblock kernlab -- an {S4} package for kernel methods in {R}.
\newblock \emph{Journal of Statistical Software}, {\bf 11}\penalty0
  (9):\penalty0 1--20.

\bibitem[Korolyuk and Borovskich(2013)]{Kor_Bor:2013}
Korolyuk, V.~S. and Borovskich, Y.~V. (2013).
\newblock \emph{Theory of {U}-statistics},   Mathematics and its Applications,
  {\bf 273}. Springer Science \& Business Media.

\bibitem[Lee(1990)]{Lee:1990}
Lee, J. (1990).
\newblock \emph{{U}-statistics: Theory and Practice},   Statistics: Textbooks
  and Monographs, {\bf 110}. Marcel Dekker, Inc.

\bibitem[Lehmann(1951)]{Leh:1951}
Lehmann, E.~L. (1951).
\newblock Consistency and unbiasedness of certain nonparametric tests.
\newblock \emph{The Annals of Mathematical Statistics}, {\bf 22}\penalty0
  (2):\penalty0 165--179.

\bibitem[Li et~al.(2017)Li, Chang, Cheng, Yang, and Poczos]{NIPS2017_6815}
Li, C.-L., Chang, W.-C., Cheng, Y., Yang, Y., and Poczos, B. (2017).
\newblock Mmd gan: Towards deeper understanding of moment matching network.
\newblock In \emph{Advances in Neural Information Processing Systems 30}, pp.
  2203--2213.

\bibitem[Long et~al.(2015)Long, Cao, Wang, and Jordan]{pmlr-v37-long15}
Long, M., Cao, Y., Wang, J., and Jordan, M. (2015).
\newblock Learning transferable features with deep adaptation networks.
\newblock In \emph{Proceedings of the 32nd International Conference on Machine
  Learning},   Proceedings of Machine Learning Research, {\bf 37}, pp. 97--105.
  PMLR.

\bibitem[Mallows(1991)]{Mal:1991}
Mallows, C.~L. (1991).
\newblock Another comment on {O}{'}{C}inneide.
\newblock \emph{The American Statistician}, {\bf 45}\penalty0 (3):\penalty0
  257.

\bibitem[Marcum(1950)]{Mar:1950}
Marcum, J. (1950).
\newblock \emph{Table of {Q} Functions}. Rand Corporation.

\bibitem[Muandet et~al.(2016)Muandet, Sriperumbudur, Fukumizu, Gretton, and
  Sch{\"o}lkopf]{Mua_Sri_Fuk:2016}
Muandet, K., Sriperumbudur, B., Fukumizu, K., Gretton, A., and Sch{\"o}lkopf,
  B. (2016).
\newblock Kernel mean shrinkage estimators.
\newblock \emph{Journal of Machine Learning Research}, {\bf 17}\penalty0
  (48):\penalty0 1--41.

\bibitem[Muandet et~al.(2017)Muandet, Fukumizu, Sriperumbudur, and
  Sch\"{o}lkopf]{MuaFukSriSch17}
Muandet, K., Fukumizu, K., Sriperumbudur, B.~K., and Sch\"{o}lkopf, B. (2017).
\newblock Kernel mean embedding of distributions : A review and beyond.
\newblock \emph{Foundations and Trends in Machine Learning}, {\bf 10}\penalty0
  (1--2):\penalty0 1--141.

\bibitem[Reddi et~al.(2015)Reddi, Ramdas, Poczos, Singh, and
  Wasserman]{Red_Ram_Poc:2015}
Reddi, S.~J., Ramdas, A., Poczos, B., Singh, A., and Wasserman, L. (2015).
\newblock On the decreasing power of kernel and distance based nonparametric
  hypothesis tests in high dimensions.
\newblock In \emph{Proceedings of the 29th AAAI Conference on Artificial
  Intelligence}, pp. 3571--3577.

\bibitem[Sch{\"o}lkopf and Smola(2002)]{Sch_Smo:2002}
Sch{\"o}lkopf, B. and Smola, A.~J. (2002).
\newblock \emph{Learning with Kernels: Support Vector Machines, Regularization,
  Optimization, and Beyond},   Adaptive Computation and Machine Learning. MIT
  Press, Cambridge (MA).

\bibitem[Sch{\"o}lkopf et~al.(1997)Sch{\"o}lkopf, Smola, and
  M{\"u}ller]{Sch_Smo_Mul:1997}
Sch{\"o}lkopf, B., Smola, A., and M{\"u}ller, K.-R. (1997).
\newblock Kernel principal component analysis.
\newblock In \emph{Proceedings of the 7th International Conference on
  Artificial Neural Networks}, pp. 583--588.

\bibitem[Sch\"{o}lkopf et~al.(1998)Sch\"{o}lkopf, Smola, and
  M\"{u}ller]{SchSmoMul98}
Sch\"{o}lkopf, B., Smola, A., and M\"{u}ller, K. (1998).
\newblock Nonlinear component analysis as a kernel eigenvalue problem.
\newblock \emph{Neural Computation}, {\bf 10}\penalty0 (5):\penalty0
  1299--1319.

\bibitem[Serfling(1980)]{Ser:1980}
Serfling, R.~J. (1980).
\newblock \emph{Approximation theorems of mathematical statistics},   Wiley
  Series in Probability and Statistics, {\bf 162}. John Wiley \& Sons.

\bibitem[Sriperumbudur et~al.(2009)Sriperumbudur, Fukumizu, Gretton, Lanckriet,
  and Sch{\"o}lkopf]{Sri_Fuk_Gre:2009}
Sriperumbudur, B., Fukumizu, K., Gretton, A., Lanckriet, G. R.~G., and
  Sch{\"o}lkopf, B. (2009).
\newblock Kernel choice and classifiability for {RKHS} embeddings of
  probability distributions.
\newblock In \emph{Advances in Neural Information Processing Systems 22}, pp.
  1750--1758.

\bibitem[Sutherland et~al.(2017)Sutherland, Tung, Strathmann, De, Ramdas,
  Smola, and Gretton]{Sut_Tun_Str:2017}
Sutherland, D.~J., Tung, H.-Y., Strathmann, H., De, S., Ramdas, A., Smola, A.,
  and Gretton, A. (2017).
\newblock Generative models and model criticism via optimized maximum mean
  discrepancy.
\newblock In \emph{Proceedings of the 5th International Conference on Learning
  Representations}.

\bibitem[Takeuchi et~al.(2006)Takeuchi, Le, Sears, and Smola]{Tak_Le_Sea:2006}
Takeuchi, I., Le, Q.~V., Sears, T.~D., and Smola, A.~J. (2006).
\newblock Nonparametric quantile estimation.
\newblock \emph{Journal of Machine Learning Research}, {\bf 7}\penalty0
  (7):\penalty0 1231--1264.

\bibitem[van~der Vaart(1998)]{Vaa:1998}
van~der Vaart, A. (1998).
\newblock \emph{Asymptotic Statistics},   Cambridge Series in Statistical and
  Probabilistic Mathematics, {\bf 3}. Cambridge University Press, Cambridge.

\bibitem[Vert et~al.(2004)Vert, Tsuda, and Sch{\"o}lkopf]{Ver_Koj_Sch:2004}
Vert, J.-P., Tsuda, K., and Sch{\"o}lkopf, B. (2004).
\newblock A primer on kernel methods.
\newblock In \emph{Kernel Methods in Computational Biology}, pp. 35--70. MIT
  Press.

\bibitem[Zhang et~al.(2017)Zhang, Filippi, Gretton, and
  Sejdinovic]{Zha_Fil_Gre:2017}
Zhang, Q., Filippi, S., Gretton, A., and Sejdinovic, D. (2017).
\newblock Large-scale kernel methods for independence testing.
\newblock \emph{Statistics and Computing}, pp. 1--18.

\end{thebibliography}

\newpage

\renewcommand{\appendixpagename}{Appendix}

\begin{appendices}
\appendixpage

In this Appendix we collect the proofs of all the results stated in the main paper and some further explanations regarding the computation of the bandwidths used in Section~\ref{sec:example}. 
It is organized as follows: 
Section~\ref{sec:proof-lemma-gap} contains the proof of Lemma~\ref{lemma:gap}, 
whereas the proofs of all the other main results can be found in Section~\ref{sec:proof-main-results}, 
and the technical lemmas are collected in Section~\ref{sec:median:additional}. 
The derivation of the cumulative distribution of~$T$ in closed form in exposed in Section~\ref{sec:cdf-derivation}. 
Section~\ref{sec:power-criterion} contains further explanation about and computation of the power criteria introduced in Section~\ref{sec:example}. 
Finally, Section~\ref{sec:proof-abs-computation} is dedicated to the proof of Prop.~\ref{prop:abs-computation}.

\medskip

\section{Proof of Lemma~\ref{lemma:gap}}
\label{sec:proof-lemma-gap}

Set $r \defeq \norm{\mu_X-\mu_Y}/5$. 
According to Chebyshev's inequality, 
\[
\begin{cases}
\proba{\norm{X-\mu_X} > r} &\leq \frac{25\trace{\Sigma_X}}{\norm{\mu_X-\mu_Y}^2} \\
\proba{\norm{Y-\mu_Y} > r} &\leq  \frac{25\trace{\Sigma_Y}}{\norm{\mu_X-\mu_Y}^2} 
\, ,
\end{cases}
\]
and this is also true for any independent copy of $X$ and $Y$.
By the union bound, there exists an event $\Omega_{XX}$ with probability greater than $1-50\trace{\Sigma_X}/\norm{\mu_X-\mu_Y}^2$ such that $\norm{X-\mu_X}\leq r$ and $\norm{X'-\mu_X}\leq r$. 
On this event,
\begin{align}
T_{XX}^{\frac{1}{2}} &= \norm{X-X'} \notag \\
& \tag{definition of $T_{XX}$} \\
& \leq \norm{X-\mu_X} + \norm{\mu_X-X'} \notag \\
& \tag{triangle inequality} \\
T_{XX}^{\frac{1}{2}} & < 2r \, . \notag 
\end{align}
The same reasoning shows that there exists an event $\Omega_{YY}$ with probability at least $1-50\trace{\Sigma_Y}/\norm{\mu_X-\mu_Y}^2$ such that $T_{YY}^{\frac{1}{2}} < 2r$.
There also exists an event $\Omega_{XY}$ with probability greater than $1-25(\trace{\Sigma_X}+\trace{\Sigma_Y})/\norm{\mu_X-\mu_Y}^2$ on which $\norm{X-\mu_X}\leq r$ and $\norm{Y-\mu_Y}\leq r$. 
On this event, 
\begin{align}
T_{XY}^{\frac{1}{2}} &= \norm{X-Y} \notag \\
& \tag{definition of $T_{XY}$} \\
& \geq \norm{\mu_X-\mu_Y} -\norm{X-\mu_X} - \norm{\mu_Y-Y} \notag \\
& \tag{triangle inequality} \\
& \geq \norm{\mu_X-\mu_Y} -2r \notag \\
& \tag{definition of $\Omega_{XY}$} \\
T_{XY}^{\frac{1}{2}} & \geq 3r  \, . \notag \\
& \tag{definition of $r$}
\end{align}
Define $\Omega\defeq \Omega_{XX}\cap \Omega_{YY}\cap\Omega_{XY}$. 
By the union bound, $\Omega$ has probability greater than $1-75\left(\trace{\Sigma_X}+\trace{\Sigma_Y}\right) / \norm{\mu_X-\mu_Y}^2$. 
Moreover, on this event, $\max\left(T_{XX}^{\frac{1}{2}},T_{YY}^{\frac{1}{2}}\right)+r < T_{XY}^{\frac{1}{2}}$, from which we deduce the result.

\qed

\section{Proofs of the main results}
\label{sec:proof-main-results}

\subsection{Proof of Prop.~\ref{prop:median:clt-ustat}}
\label{sec:median:proof-clt-ustat}


The proof consists of three steps.

\paragraph{Step 1: Decomposition of $U_n$.}

We begin by decomposing~$U_n$. 
To this extent, define
\[
A_n = \binom{\alpha n}{2}^{-1} \sum_{1\leq i < j\leq \alpha n} h(X_{n,i},X_{n,j})
\, , \quad
B_n = \binom{(1-\alpha)n}{2}^{-1} \sum_{\alpha n < i < j\leq n} h(X_{n,i},X_{n,j})
\, ,
\]
and
\[
C_n = \frac{1}{\alpha n (1-\alpha)n}\sum_{\substack{1\leq i\leq \alpha n \\ \alpha n < j \leq n}} h(X_{n,i},X_{n,j})
\, .
\]
Note that in~$C_n$ there is no term such that $h(X_{n,i},X_{n,j})$ with $1\leq j\leq \alpha n$ and $\alpha n < i\leq n$, since the sum in~$U_n$ is prescribed to $i<j$.
Simple algebra shows that
\begin{align}
\label{eq:median:decomp-un}
U_n &= \alpha^2 A_n + (1-\alpha)^2 B_n + 2\alpha(1-\alpha)C_n + s_n. 
\end{align}
where $s_n = \frac{\alpha(1-\alpha)}{n-1}A_n - \frac{\alpha(1-\alpha)}{n-1}B_n - \frac{2\alpha(1-\alpha)}{n-1} C_n$.
According to Lemma~\ref{lemma:median:variance-ABC} (see Appendix~\ref{sec:median:additional}), the variance of each of $A_n$, $B_n$ and~$C_n$ is $\bigo{1/n}$, and hence $s_n$ converges to $0$ in probability at speed~$n^{-2}$. 
Therefore we will focus on the terms of Eq.~\eqref{eq:median:decomp-un} other than $s_n$.

\paragraph{Step 2: Decompositions of $A_n$, $B_n$ and $C_n$.}

The next step is to obtain the $H$-decomposition~\citep[Section~1.6]{Lee:1990} of~$A_n$ and~$B_n$.
Let us detail this process for~$A_n$.
We set 
\begin{eqnarray*}
\theta_A &=& \expec{h(X,X')}, \\
h_A(x) &=& \expec{h(x,X')}-\theta_A, \\
g_A(x,y) &=& h(x,y)-h_A(x)-h_A(y)-\theta_A \\
L_A &=& \frac{2}{\alpha n}\sum_{1\leq i\leq \alpha n}h_A(X_{n,i}), \\
R_A &=& \binom{\alpha n}{2}^{-1}\sum_{1\leq i<j\leq \alpha n}g_A(X_{n,i},X_{n,j}).
\end{eqnarray*}
Then it is possible to write $A_n$ as~\citep[Th.~1]{Lee:1990} 
\begin{equation*}
A_n=\theta_A +L_A+R_A
\, .
\end{equation*}
%
%
A totally analogous statement holds for~$B_n$: Define
\begin{eqnarray*}
\theta_B &=& \expec{h(Y,Y')},\\ 
h_B(y) &=& \expec{h(y,Y')}-\theta_B,\\
g_B(x,y) &=& h(x,y)-h_B(x)-h_B(y)-\theta_B, \\
L_B &=& \frac{2}{(1-\alpha) n}\sum_{\alpha n < i \leq n}h_B(X_{n,i}), \\
R_B &=& \binom{(1-\alpha) n}{2}^{-1}\sum_{\alpha n < i < j \leq n}g_B(X_{n,i},X_{n,j}).
\end{eqnarray*}
Then $B_n$ can be written as 
\[
B_n=\theta_B +L_B+R_B
\, .
\]

We decompose~$C_n$ in the same fashion, the only difference being the appearance of a second term in the linear part.
Namely, set 
\begin{eqnarray*}
\theta_C &=& \expec{h(X,Y)}, \\
h_{C,1}(x) &=& \expec{h(x,Y)}, \\
 h_{C,2}(y) &=& \expec{h(X,y)}, \\
 g_C(x,y) &=& h(x,y) - h_{C,1}(x) - h_{C,2}(y) -\theta_C, \\
 L_C &=& \frac{1}{\alpha n}\sum_{i=1}^{\alpha n}h_{C,1}(X_{n,i})+\frac{1}{(1-\alpha)n}\sum_{i=\alpha n +1}^n h_{C,2}(X_{n,i}), \\
 R_C &=& \frac{1}{\alpha n(1-\alpha)n} \sum_{\substack{1\leq i\leq \alpha n \\ \alpha n < j\leq n}} g_C(X_{n,i},X_{n,j})
 \, .
\end{eqnarray*}
Then we have
\[
C_n=\theta_C+L_C+R_C
\, .
\]
%
%
Since $\theta$ defined in Eq.~\eqref{eq:def-theta} can be written as 
\[
\theta = \alpha^2\theta_A + (1-\alpha)^2\theta_B + 2\alpha(1-\alpha)\theta_C
\, ,
\]
we have by Eq.~\eqref{eq:median:decomp-un} and the above expressions for $A_n$, $B_n$ and $C_n$
\begin{eqnarray}
&& \sqrt{n}(U_n-\theta) \nonumber \\
&=& \sqrt{n}\left[ \alpha^2 (L_A + R_A) + (1-\alpha)^2 (L_B + R_B) + 2\alpha(1-\alpha) (L_C + R_C) + s_n \right] \nonumber \\
&=& \sqrt{n}\left[\alpha^2 L_A+(1-\alpha)^2 L_B  +2\alpha(1-\alpha)L_C\right] + \sqrt{n} (r_n + s_n) \label{eq:median:decomp-clt}
\, ,
\end{eqnarray}
where $r_n = \alpha^2 R_A + (1-\alpha)^2 R_B + 2\alpha (1-\alpha) R_C$.
According to Lemma~\ref{lemma:median:variance-RABC} in Appendix \ref{sec:median:additional}, the variance of each of $R_A$, $R_B$ and $R_C$ is of order $n^{-2}$.
As mentioned earlier $s_n$ is also of order $n^{-2}$.
Therefore we have $\sqrt{n}(r_n+s_n) \cvproba 0$ as $n \to \infty$.
This shows that we only need to focus on the first term of Eq.~\eqref{eq:median:decomp-clt}.

\paragraph{Step 3: Regrouping of the terms in Eq.~\eqref{eq:median:decomp-clt}.}

We now regroup the first term in Eq.~\eqref{eq:median:decomp-clt} that belong to the same segment (\emph{i.e.}, $1\leq i\leq \alpha n$ or  $\alpha n < i \leq n$).
That is, a simple calculation yields that
\begin{eqnarray}
&& \sqrt{n}\left[\alpha^2 L_A+(1-\alpha)^2 L_B +2\alpha(1-\alpha)L_C\right] \label{eq:regroup} \\ 
&=& \sqrt{n} \left[\frac{2}{n}\sum_{1\leq i\leq \alpha n} \left( \alpha h_A(X_{n,i}) + (1-\alpha) h_{C,1}(X_{n,i}) \right)  \right. \nonumber \\
&& \left. +\frac{2}{n}\sum_{\alpha n < i \leq n} \left( \alpha h_{C,2}(X_{n,i})  (1-\alpha) h_B(X_{n,i}) \right) \right]. \nonumber
\end{eqnarray}
For any $ 1\leq i\leq \alpha n$, define
\[
Z_{n,i}^{(1)} = \alpha h_A(X_{n,i}) + (1-\alpha)h_{C,1}(X_{n,i})
\, .
\]
Since $h\in L^2(P)\cap L^2(Q) \times L^2(P)\cap L^2(Q)$ and $X_{n,i} \sim X$,
\[
\sigma_1^2 := \var{Z_{n,i}^{(1)}}=\var{\alpha h_A(X) + (1-\alpha)h_{C,1}(X)}
\]
is finite and does not depend on~$i$.
Note also the $(Z_{n,i}^{(1)})_{i=1}^{\alpha n}$ are independent.
Therefore the Lindeberg's condition is satisfied, and hence we can apply the central limit theorem for triangular arrays of independent random variables~\citet[Th.~27.2]{Bil:2012} to obtain
\[
\frac{1}{\sqrt{\alpha n}}\sum_{i=1}^{\alpha n} Z_{n,i}^{(1)} \cvlaw \gaussian{0}{\sigma_1^2}
\, \quad (n \to \infty).
\]
In a similar fashion, for any $ \alpha n < i \leq n$, let 
\[
Z_{n,i}^{(2)}\defeq \alpha h_{C,2}(X_{n,i})+(1-\alpha)h_B(X_{n,i})
\, ,
\] 
and 
\[
\sigma_2^2\defeq \var{Z_{n,i}^{(2)}} = \var{ \alpha h_{C,2}(Y)+(1-\alpha)h_B(Y) }
\, .
\]
Then we have
\[
\frac{1}{\sqrt{(1-\alpha)n}} \sum_{i=\alpha n + 1}^n  Z_{n,i}^{(2)}\cvlaw \gaussian{0}{\sigma_2^2}
\, \quad (n \to \infty)
\, .
\]
The two previous sums are independent, and thus by L\'evy's theorem and Eq.~\eqref{eq:regroup} we have
\begin{eqnarray*}
&& \sqrt{n}\left[\alpha^2 L_A+(1-\alpha)^2 L_B +2\alpha(1-\alpha)L_C\right] \\
&=& \sqrt{n} \left[\frac{2}{n}\sum_{1\leq i\leq \alpha n} Z_{n,i}^{(1)} +   \frac{2}{n}\sum_{\alpha n < i \leq n} Z_{n,i}^{(2)} \right] \\
&=& \frac{2 \sqrt{\alpha} }{\sqrt{\alpha n} }\sum_{1\leq i\leq \alpha n} Z_{n,i}^{(1)} +   \frac{2 \sqrt{1-\alpha} }{ \sqrt{(1-\alpha)n} }\sum_{\alpha n < i \leq n} Z_{n,i}^{(2)}  
\cvlaw \gaussian{0}{\sigma^2} \, \quad (n \to \infty) ,
\end{eqnarray*}
where $\sigma^2 = \sigma_1^2 + \sigma_2^2$, which can be shown to be equal to Eq.~\eqref{eq:median:def-variance} by a straightforward computation.
Since the remainder term in Eq.~\eqref{eq:median:decomp-clt} converges in probability to~$0$, we can conclude \emph{via\/} Slutsky's Lemma \citep[Th.~2.7]{Vaa:1998}.
\qed

\subsection{Proof of Prop.~\ref{prop:median:clt-median-heuristic}} 
\label{sec:proof-clt-median-heuristic}

Set $t\in\Reals$.
The general idea of the proof is to rewrite statements about the event $\bigl\{\sqrt{n} \left(\Sqmedh_n-\Medvalue\right)\leq t\bigr\}$ as statements about a sum of $U$-statistics.
We will then control these $U$-statistics with Prop.~\ref{prop:median:clt-ustat} for  conveniently chosen~$h$, and conclude with Slutsky's Lemma.
Throughout this proof, we only suppose that $p\in (0,1)$ to emphasize that Prop.~\ref{prop:median:clt-median-heuristic} can be extended to \emph{any\/} quantile, not only the median.
Therefore in this proof we define $m$ by 
\[
m := \hat{F}(p)
\, .
\]
%
Note that the median case corresponds to $p = 1/2$.

We use the property of the generalized inverse to obtain
\[
\bigl\{\sqrt{n}\left(\Sqmedh_n-\Medvalue\right)\leq t\bigr\} = \biggl\{p\leq \Ecdf_n\left(\Medvalue + \frac{t}{\sqrt{n}}\right)\biggr\}
\, ,
\]
and rewrite this event as
\begin{equation}
\label{eq:median:rhs}
\left\{\sqrt{n}\left(\Ecdf_n\left(\Medvalue + \frac{t}{\sqrt{n}}\right) - \cdf{\Medvalue+\frac{t}{\sqrt{n}}}\right)\geq \sqrt{n}\left(p-\cdf{\Medvalue +\frac{t}{\sqrt{n}}}\right)\right\}
\, .	
\end{equation}
Since~$\Cdf$ is differentiable in~$\Medvalue$, the right-hand side of Eq.~\eqref{eq:median:rhs} converges towards $-t\Cdf'(\Medvalue)$ by Taylor expansion.
From Prop.~\ref{prop:median:clt-ustat}, it is also true that 
\[
\sqrt{n}\left(\Ecdf_n(\Medvalue) - \cdf{\Medvalue}\right) \cvlaw \gaussian{0}{\sigma^2}
\, .
\]
Therefore, if we manage to prove that
\begin{equation}
\label{eq:median:key}
\sqrt{n}\left[ \left(\Ecdf_n\left(\Medvalue + \frac{t}{\sqrt{n}}\right) - \Ecdf_n(\Medvalue)\right) -\left(\cdf{\Medvalue+\frac{t}{\sqrt{n}}}-\cdf{\Medvalue}\right)\right] \cvproba 0
\, ,
\end{equation}
Then Eq.~\eqref{eq:median:clt-median-heuristic} will follow by Slutsky's Lemma.

Define $h(x,y) = \indic{\Medvalue < \norm{x-y}^2 \leq \Medvalue +\frac{t}{\sqrt{n}}}$.
Then, with the notations used in the proof of Prop.~\ref{prop:median:clt-ustat}, Eq.~\eqref{eq:median:key} reads
\[
\sqrt{n}(U_n-\theta) \cvproba 0
\, .
\]
We dispose of the remainder terms as in the proof of Prop.~\ref{prop:median:clt-ustat}, thus we are left to show that
\begin{equation}
\label{eq:median:key-translation}
\sqrt{n}\left[\alpha^2 (A_n-\theta_A) +(1-\alpha)^2(B_n-\theta_B) +2\alpha(1-\alpha)(C_n-\theta_C)\right] \cvproba 0
\, .
\end{equation}
Let us focus on the first term of the previous display, which can be written
\[
\alpha^2\sqrt{n}\left(A_n-\theta_A\right) = \alpha^2\sqrt{n}\left(L_A+R_A\right)
\, .
\]
Once again, we use Lemma~\ref{lemma:median:variance-RABC} to get rid of~$R_A$.
The linear term is slightly more tedious to analyze.
Recall that $\expec{h_A(X)}=0$.
Thanks to Jensen's inequality,
\begin{align*}
\var{h_A(X)} &= \expecunder{\left(\expecunder{\indic{\Medvalue <\norm{X-X'}^2\leq \Medvalue +\frac{t}{\sqrt{n}}}}{X'}^2\right)}{X} \\
&\leq \proba{\Medvalue < \norm{X-X'}^2\leq \Medvalue + \frac{t}{\sqrt{n}}}
\, .
\end{align*}
We recognize
\[
\cdf{\Medvalue+\frac{t}{\sqrt{n}}}-\cdf{\Medvalue}
\, ,
\]
which goes to~$0$ when $n\to\infty$, since we assumed~$\Cdf$ to have a derivative in~$\Medvalue$.
Furthermore, by independence of the~$X_{n,i}$,
\[
\var{\sqrt{n}L_A} = 4\var{h_A(X)} \underset{n\to\infty}{\longrightarrow} 0
\, .
\]
A similar reasoning applies to the other terms in Eq.~\eqref{eq:median:key-translation}, and the proof is concluded.

\section{Technical lemmas} 
\label{sec:median:additional}

In this section, we state and prove the technical results that are needed in the proofs of the main results.
Recall that we denote by $\card{A}$ the cardinality of any finite set $A$.

\begin{lemma}[\textbf{Controlling the variance of $A_n$, $B_n$, and $C_n$}]
	\label{lemma:median:variance-ABC}
	Let $A_n$, $B_n$, and $C_n$ be defined as in the proof of Prop.~\ref{prop:median:clt-ustat}. 
	Then $\var{A_n}$, $\var{B_n}$, and $\var{C_n}$ are $\bigo{n^{-1}}$.
\end{lemma}

The proof is standard in $U$-statistics~\citep{Lee:1990}, we provide a version thereof for completeness' sake.
	
\begin{proof}
We set $m\defeq \alpha n$ in this proof.
Recall that 
\[
A_n=\binom{m}{2}^{-1} \sum_{1\leq i < j\leq m} h(X_{n,i},X_{n,j})
\, ,
\]
Define $h_{i,j}=h(X_{n,i},X_{n,j})$, thus $\Expec h_{i,j}=\theta_A$ and $\Expec A_n = \theta_A$.
Let us turn to the computation of $\expec{A_n^2}$, that is, 
\[
\expec{A_n^2} = \binom{m}{2}^{-2}\sum_{a<b}\sum_{c<d} \expec{h_{a,b}h_{c,d}}
\, ,
\]
where $a,b,c$ and $d$ range from~$1$ to $m$.
In this last display, there are three possibilities for the indices in the summation that we detail below.
\begin{itemize}
	\item
	The indices $a,b,c$ and $d$ are all distinct. 
	There are $\binom{m}{4}\binom{4}{2} = \frac{m^4}{4}+\bigo{m^3}$ ways to choose such indices, that is, $\binom{m}{4}$ ways to choose the location of the~$4$ indices among the~$m$ possible locations, then $\binom{4}{2}$ choices for, say, $a<b$, and only one possibility left for $c<d$.
	
	\item
	One of the indices is common, that is, $\card{\{a,b\}\cap\{c,d\}}=1$. 
	There are $6\binom{m}{3}=\bigo{m^3}$ ways to do so.
	
	\item 
	Both indices are equal, that is, $a=c$ and $b=d$. 
	There are $\binom{m}{2}=\bigo{m^2}$ ways to do so.
\end{itemize}
Note that when $a,b,c$ and $d$ are all distinct, $\expec{h_{a,b}h_{c,d}}=\expec{h_{a,b}}\expec{h_{c,d}}$ by independence of the~$X_{n,i}$s.
Thus 
\begin{align*}
\expec{A_n^2} &= \binom{m}{2}^{-2}\sum_{\substack{a<b \\ c<d \\ \neq}}\expec{h_{a,b}}\expec{h_{c,d}} +\bigo{m^{-1}}
\, ,
\end{align*}
where we sum on \emph{distinct} indices.
On the other hand,
\[
\expec{A_n}^2 = \binom{m}{2}^{-2} \sum_{\substack{a < b \\ c < d}} \Expec h_{a,b} \Expec h_{c,d}
\, .
\]
By the same combinatorial argument, the terms corresponding to intersecting sets of indices are at most $\bigo{n^3}$ and we have
\[ 
\expec{A_n}^2 = \binom{m}{2}^{-2} \sum_{\substack{a < b \\ c < d \\ \neq}} \Expec h_{a,b} \Expec h_{c,d} + \bigo{m^{-1}}
\, .
\]
Since $m=\alpha n$, $\bigo{m^{-1}}=\bigo{n^{-1}}$ and we can conclude for $A_n$:
\[
\var{A_n} = \expec{A_n^2} - \expec{A_n}^2 =\bigo{n^{-1}}
\, .
\]
The same proof transfers readily for $B_n$ and $C_n$.
\end{proof}

\begin{lemma}[\textbf{Controlling the variance of $R_A$, $R_B$, and $R_C$}]
	\label{lemma:median:variance-RABC}
	Let $R_A$, $R_B$, and $R_C$ be as in the proof of Prop.~\ref{prop:median:clt-median-heuristic}. 
	Then $\var{R_A}$, $\var{R_B}$, and $\var{R_C}$ are of order $\bigo{n^{-2}}$.
\end{lemma}

\begin{proof}
Recall that
\[
R_A = \binom{m}{2}^{-1}\sum_{1\leq i<j\leq m} g_{i,j}
\, ,
\]
with $g_{i,j} = g_A(X_{n,i},X_{n,j})$.
By definition of $g_A$, it holds that $\Expec g_{i,j}=0$ for any $i\neq j$, thus $\Expec R_A=0$.
Hence to control the variance of~$R_A$, we just need to compute $\expec{R_A^2}$.
As in the proof of Lemma~\ref{lemma:median:variance-ABC}, we have
\[
\expec{R_A^2} = \binom{m}{2}^{-2} \sum_{a<b}\sum_{c<d}\expec{g_{a,b}g_{c,d}}
\, .
\]
Note that $\expec{g_{a,b}g_{c,d}}=0$ whenever $a,b,c$ and $d$ are all distinct.
But a straightforward computation also shows that $\expec{g_{a,b}g_{a,c}}=0$ for any distinct $a,b,c$.
Thus the previous display reduces to
\[
\expec{R_A^2} = \binom{m}{2}^{-2} \sum_{\substack{a<b\\ c<d \\ \star}}\expec{g_{a,b}g_{c,d}}
\, ,
\]
where $\star$ denotes that we sum on indices such that $\card{\{a,b\}\cap\{c,d\}}\geq 2$.
As we have seen in the proof of Lemma~\ref{lemma:median:variance-ABC}, there are only $\bigo{m^2}$ such possibilities, and we can conclude.
\end{proof}

\section{Derivation of the cumulative distribution function of T}
\label{sec:cdf-derivation}

In this section, we derive the population cumulative distribution of the pairwise distances in closed-form in both scenarios considered in Section~\ref{sec:example}. 
We first introduce some additional notations. 
We let~$\gamma$ be the \emph{lower} incomplete gamma function:
\[
\gamma(a,x) = \int_0^x t^{a-1}\exps{-t}\diff t
\, ,
\]
and $Q_M$ the Marcum $Q$-function \citep{Mar:1950}, defined by
\[
Q_M(a,b) = \int_b^{+\infty} x \left(\frac{x}{a}\right)^{M-1} \exp{-\frac{x^2+a^2}{2}}I_{M-1}(ax) \diff x
\, ,
\]
where $I_{M-1}$ is the modified Bessel function of order $M-1$ \citep{Abr_Ste:1964}.

\paragraph{Change in the mean.}

In the {\bf Mean} scenario, $T_{XX}$ and $T_{YY}$ have the law of $2\chi_1^2$, whereas $T_{XY}$ has the distribution of a non-central chi-squared random variable $\chi_1^2(\mu/2)$. 
It is well-known that the cumulative distribution function of a chi-squared random variable with one degree of freedom is given by $\frac{\gamma\left(\frac{1}{2},\frac{\cdot}{2}\right)}{\Gamma\left(\frac{1}{2}\right)}$, whereas that of a non-central chi-squared with non-centrality parameter $\lambda$ is given by $1-Q_{1/2}\left(\sqrt{\lambda},\sqrt{\cdot}\right)$. 
Thus, according to the definition of~$T$ as a mixture, we have 
\[
F_T(t) = \left[\alpha^2+(1-\alpha)^2\right] \frac{\gamma\left(\frac{1}{2},\frac{t}{4}\right)}{\Gamma\left(\frac{1}{2}\right)} + 2\alpha(1-\alpha) \left(1-Q_{\frac{1}{2}}\left(\frac{\mu}{\sqrt{2}},\sqrt{\frac{t}{2}}\right)\right)
\, .
\]
We plot $F_T$ for a few values of $\mu$ in Figure~\ref{fig:cdf-example}.
Finding the value of the population median of~$T$ is equivalent to solving $F_T(t)=1/2$ for $t$. 
Since $F_T$ is an increasing function, this can be achieved by dichotomy, or, even faster, \emph{via} Newton method.
A straightforward computation yields the derivative of~$F_T$, 
\[
f_T(t) = \frac{\left[\alpha^2+(1-\alpha)^2\right]\exps{-t/4} + 2\alpha(1-\alpha)\exps{-(\mu^2+t)/4}\cosh \left(\frac{\mu\sqrt{t}}{2}\right)}{2\sqrt{\pi t}}
\, .
\]
\begin{figure}[t!]
\centering 
\includegraphics[width=0.49\linewidth]{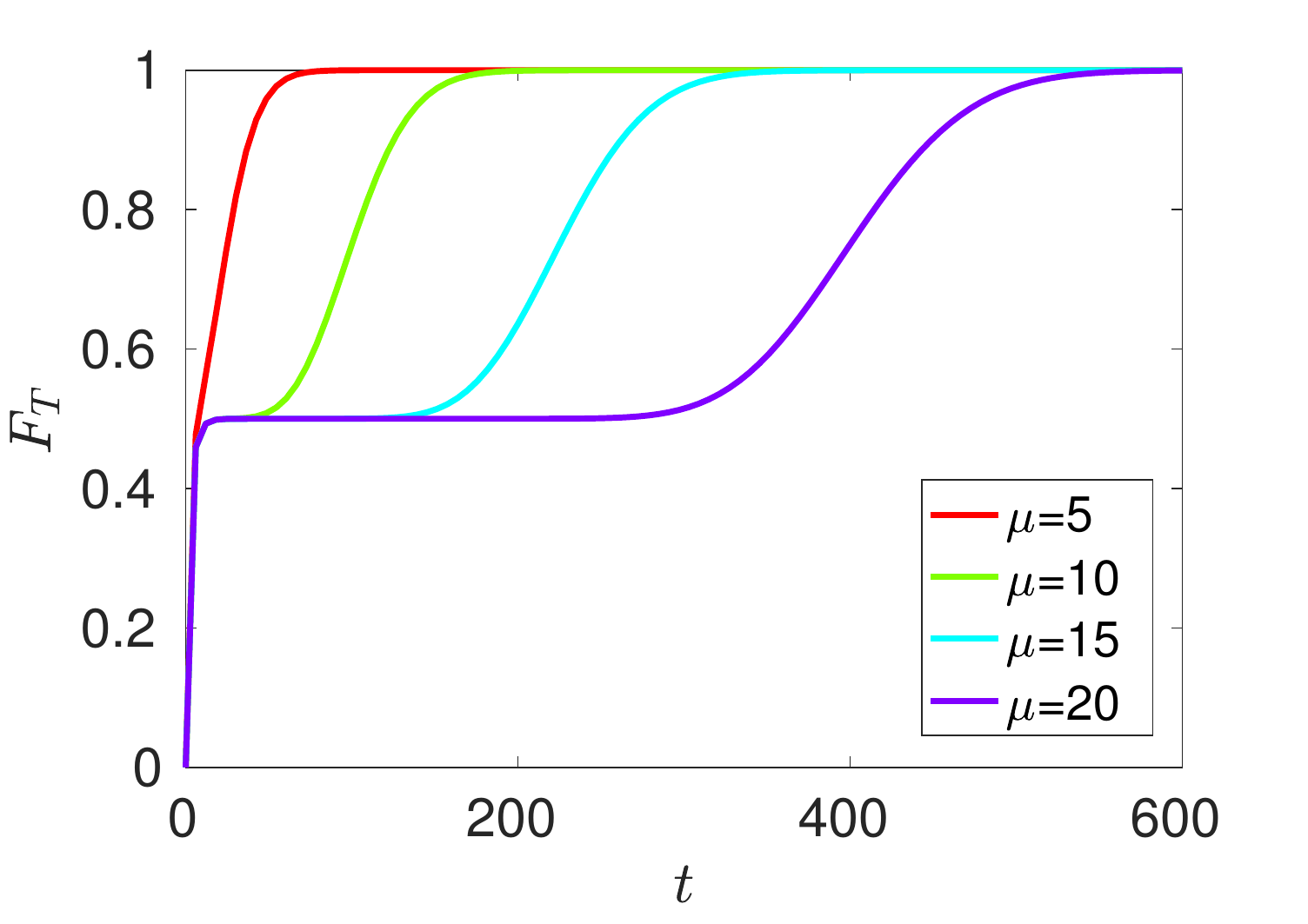}
\hspace{-0.5cm}
\includegraphics[width=0.49\linewidth]{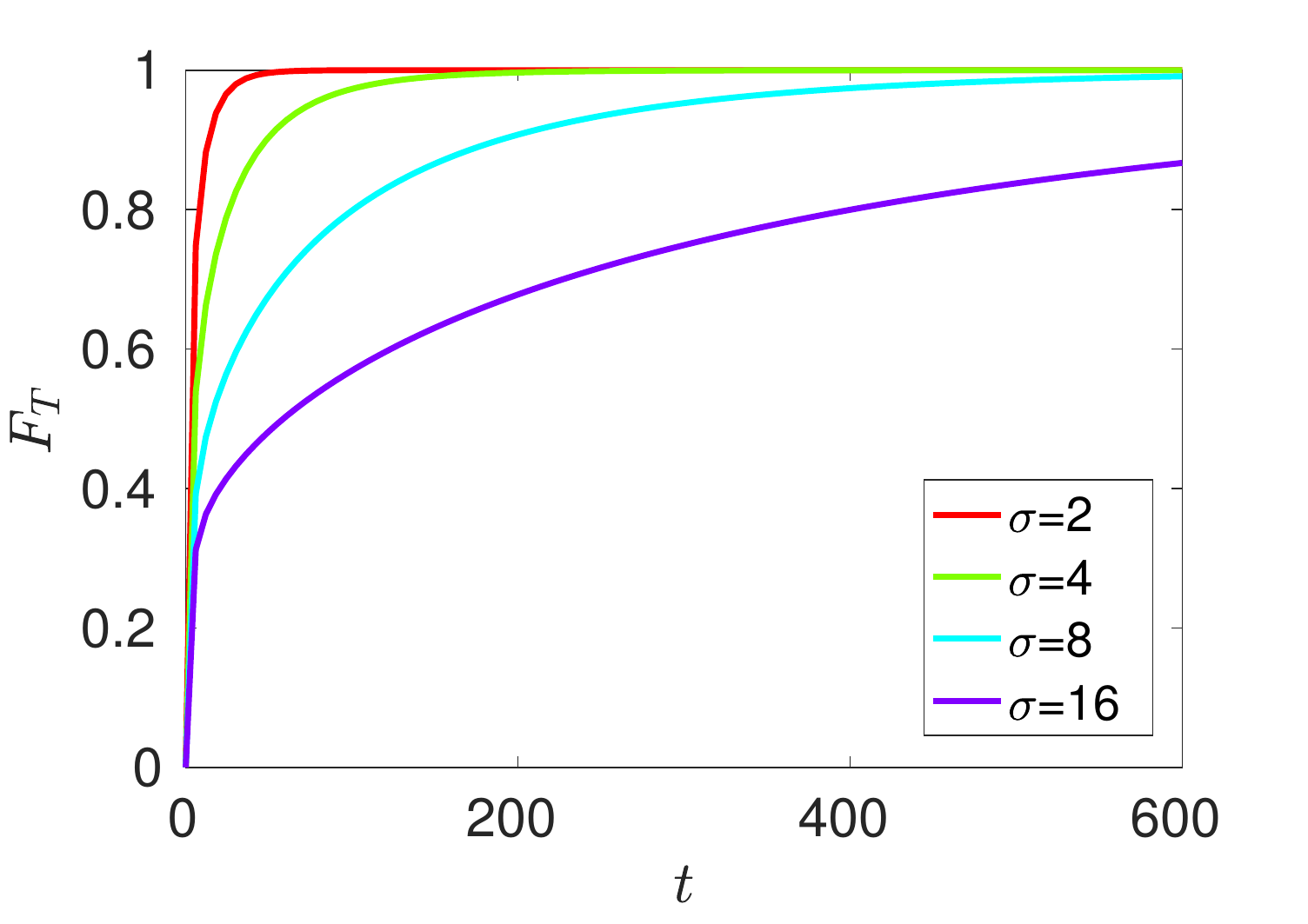}
\caption{\label{fig:cdf-example}In this figure, we plot the cumulative distribution function of~$T$. 
    \emph{Left panel:} change in the mean scenario. 
    \emph{Right panel:} change in the variance.
    Note that $F_T$ is becoming very ``flat'' when $\mu$ increases in the left panel, leading to numerical problems when solving $F_T(t)=1/2$. 
}
\end{figure}

\paragraph{Change in the variance.}

In the {\bf Var} scenario, $T_{XX}$ has the law of $2\chi_1^2$, $T_{YY}$ has the law of $2\sigma^2\chi_1^2$, and $T_{XY}$ has the law of $(1+\sigma^2)\chi_1^2$. 
By the same reasoning, we obtain the cumulative distribution function of~$T$, 
\[
F_T(t) = \frac{\alpha^2\gamma\left(\frac{1}{2},\frac{t}{4}\right)+(1-\alpha)^2\gamma\left(\frac{1}{2},\frac{t}{4\sigma^2}\right) + 2\alpha(1-\alpha)\gamma\left(\frac{1}{2},\frac{t}{2(\sigma^2+1)}\right)}{\Gamma\left(\frac{1}{2}\right)}
\, .
\]
We plot $F_T$ for a few values of $\sigma$ in Figure~\ref{fig:cdf-example}.
A straightforward computation yields
\[
f_T(t) = \frac{\alpha^2\exp{\frac{-t}{4}}}{2\sqrt{\pi t}} + \frac{(1-\alpha)^2\exp{\frac{-t}{4\sigma^2}}}{2\sigma\sqrt{\pi t}} + \frac{2\alpha(1-\alpha)\exp{\frac{-t}{2(\sigma^2+1)}}}{\sqrt{2(\sigma^2+1)\pi t}}
\, .
\]

\section{Power criterion for kernel two-sample test}
\label{sec:power-criterion}

In this section, we set $z=(x,y)$ with $x\sim X$ and $y\sim Y$, where it should be clear that~$X$ and~$Y$ depend on the scenario that we are investigating. 
As before,~$x'$ and~$y'$ denote independent copies of~$x$ and~$y$, thus~$z'$ is an independent copy of~$z$. 
Note that the proportion of~$X$ and~$Y$ is fixed to $\alpha=1/2$. 
Recall that we focus on $k_{\nu}(x,y)=\exp{-(x-y)^2/(2\nu^2)}$, the Gaussian kernel with bandwidth $\nu$, that we denote by~$k$ for concision. 
We also set
\[
h(z,z') \defeq \kernel{x}{x'} + \kernel{y}{y'} - \kernel{x}{y'} - \kernel{x'}{y}
\, .
\]

According to \citet{Gre_Sej_Str:2012}, the asymptotic probability of a type II error for the linear-time test statistic $\linMMD$ at level $a$ is given by
\[
\Phi\left(\Phi^{-1}\left(1-a\right) -\frac{\MMD^2\sqrt{n}}{\sigma_{\ell}\sqrt{2}}\right)
\, ,
\]
where $\MMD^2$ is the maximum mean discrepancy between $P$ and $Q$ and $\sigma_{\ell}^2$ is an asymptotic variance term given by \citep[Section~6]{Gre_Bor_Ras:2012}
\begin{equation}
\label{eq:lin-variance}
\sigma_{\ell}^2 \defeq 2\left(\expecunder{h(z,z')^2}{z,z'} - \left(\expecunder{h(z,z')}{z,z'}\right)^2\right)
\, .
\end{equation}
Thus a natural way to select the kernel---in our case the bandwidth---is to maximize the \emph{power ratio criterion}
\[
\ratioLin = \ratioLin(P,Q,\Kernel) \defeq \frac{\MMD^2}{\sigma_{\ell}}
\, ,
\]
keeping in mind that both the maximum mean discrepancy and $\sigma_{\ell}$ are depending on $P$, $Q$, and $\Kernel$.

For the quadratic-time test statistic $\quadMMD$, the picture is not so clear since the null distribution is more complicated and the type II error depends on a quantile of this distribution, namely
\[
\frac{c_{a}}{\sqrt{n}} - \frac{\MMD^2 \sqrt{n}}{\sigma_u\sqrt{2}}
\, ,
\]
where $\sigma_u^2$ is an asymptotic variance term given by \citep[Section~6]{Gre_Bor_Ras:2012}
\begin{equation}
\label{eq:quad-variance}
\sigma_u^2 \defeq 4\left(\expecunder{\left(\expecunder{h(z,z')}{z'}\right)^2}{z} - \left(\expecunder{h(z,z')}{z,z'}\right)^2\right)
\, .
\end{equation}
Nevertheless, we choose the bandwidth that maximizes the following power ratio criterion
\[
\ratioQuad = \ratioQuad(P,Q,\Kernel) \defeq \frac{\MMD^2}{\sigma_u}
\, .
\]
We now proceed to compute the ratio $\ratioLin$ and $\ratioQuad$ in both the change of mean and the change of variance scenarios.
Note that, though the MMD expressions are well-known, the closed-form expressions we obtain for the variance terms are new up to the best of our knowledge.

\paragraph{Change in the mean.}

In this case, the MMD is given by
\[
\expecunder{h(z,z')}{z,z'} = \frac{2\nu}{\sqrt{\nu^2+2}}\left(1-\exp{\frac{-\mu^2}{2(\nu^2+2)}}\right)
\, .
\]
See, for instance, \citet{Red_Ram_Poc:2015}. 
We now turn to the variance terms.
The second term in Eq.~\eqref{eq:lin-variance} and Eq.~\eqref{eq:quad-variance} is the MMD, given by the previous display. 
The remaining computation is, in both the linear-time and the quadratic-time case, a lengthy derivation that uses intensively the following lemma:

\begin{lemma}[\textbf{Gaussian integral computation}]
\label{lemma:master}
Let $a,b,c,d\in \Reals$ with $b,d >0$.
Then
\[
\frac{1}{\sqrt{2\pi}} \int \exp{\frac{-(x-a)^2}{b} + \frac{-(x-c)^2}{d}} \diff x = \sqrt{\frac{bd}{2(b+d)}} \exp{\frac{-(a-c)^2}{b+d}}
\, .
\]
\end{lemma}
\begin{proof}
We write 
\[
\frac{-(x-a)^2}{b}+\frac{-(x-c)^2}{d} = \frac{1}{bd}\left\{(b+d)x^2+(2ad+2bc)x-(a^2d+bc^2)\right\}
\, ,
\]
and then use the well-known identity
\[
\int_{-\infty}^{+\infty} \exps{-\alpha x^2 + \beta x +\gamma} \diff x = \sqrt{\frac{\pi}{\alpha}}\exps{\frac{\beta^2}{4\alpha}+\gamma}
\, .
\]
\end{proof}

For the linear-time statistic, we obtain
\begin{align*}
\expecunder{h(z,z')^2}{z,z'} &= \frac{2\nu}{\sqrt{\nu^2+4}}\left(1+\exp{\frac{-\mu^2}{\nu^2+4}}\right) \\
&+ \frac{2\nu^2}{\nu^2+2} \left(1+\exp{\frac{-\mu^2}{\nu^2+2}}\right) \\
&- \frac{8\nu^2}{\sqrt{(\nu^2+1)(\nu^2+3)}}\exp{\frac{-(\nu^2+2)\mu^2}{2(\nu^2+1)(\nu^2+3)}} \, .
\end{align*}
For the quadratic-time statistic, we obtain
\begin{align*}
\expecunder{\left(\expecunder{h(z,z')}{z'}\right)^2}{z} &= \frac{2\nu^2}{\sqrt{(\nu^2+1)(\nu^2+3)}}\left(1+\exp{\frac{-\mu^2}{\nu^2+3}}\right) \\
&+ \frac{2\nu^2}{\nu^2+2} \left(1+\exp{\frac{-\mu^2}{\nu^2+2}}\right) \\
&- \frac{4\nu^2}{\sqrt{(\nu^2+1)(\nu^2+3)}}\exp{\frac{-(\nu^2+2)\mu^2}{2(\nu^2+1)(\nu^2+3)}} \\
&- \frac{4\nu^2}{\nu^2+2}\exp{\frac{-\mu^2}{2(\nu^2+2)}}
\, .
\end{align*}
In Figure~\ref{fig:ratios-mean}, we plot both $\ratioLin$ and $\ratioQuad$ for various values of the change parameter~$\mu$. 

\begin{figure}[t!]
\centering 
\includegraphics[width=0.49\linewidth]{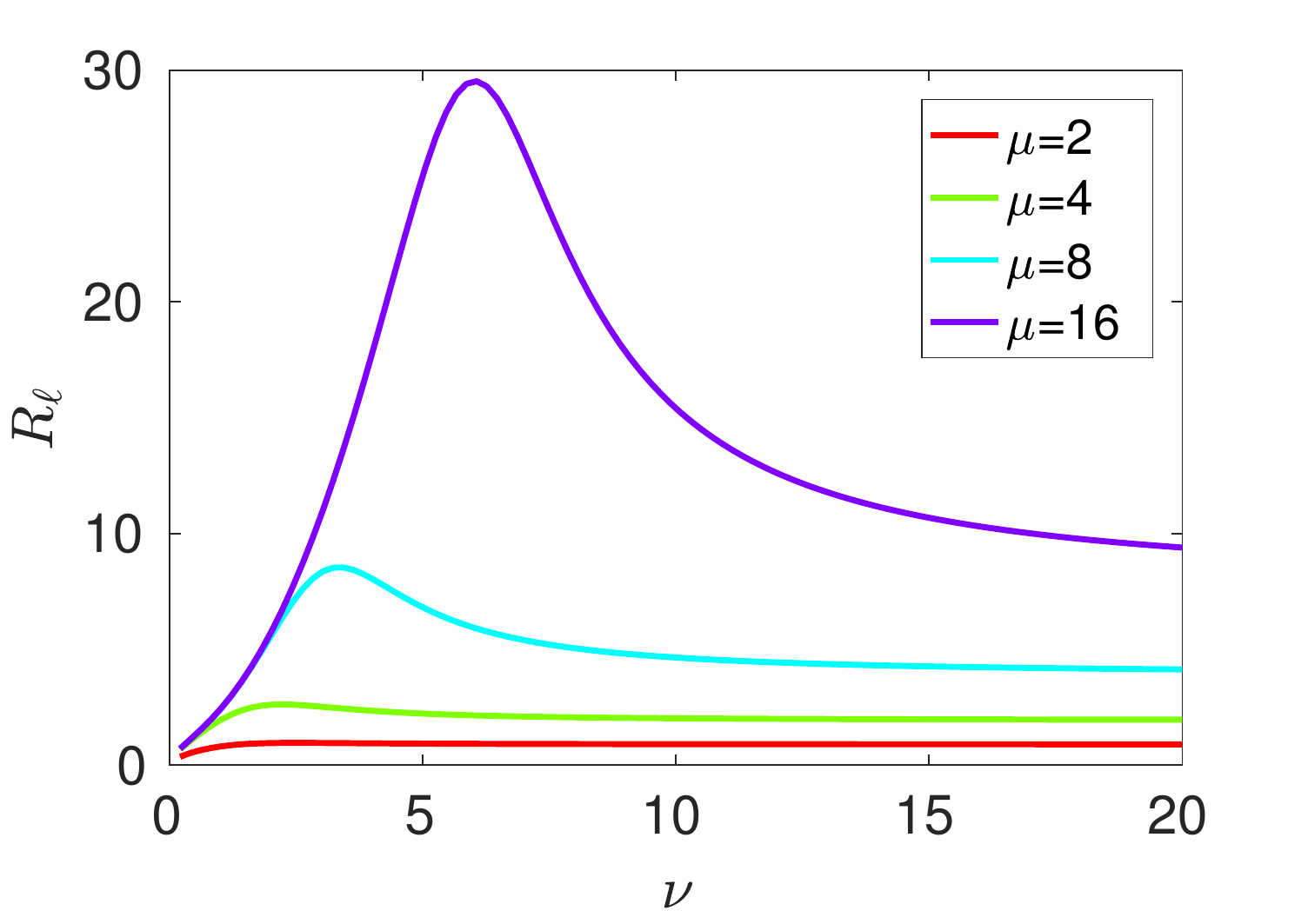}
\hspace{-0.5cm}
\includegraphics[width=0.49\linewidth]{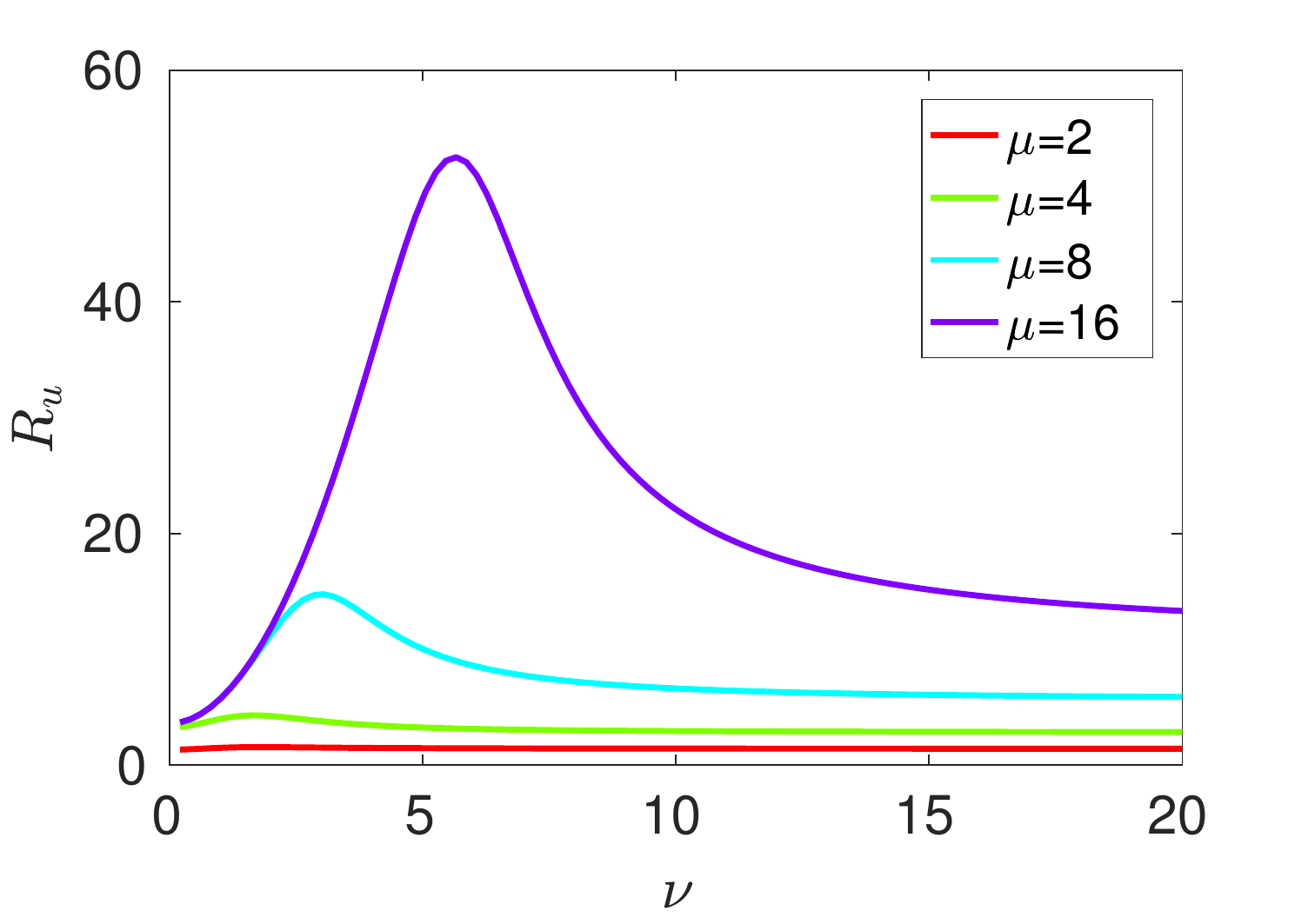}
\caption{\label{fig:ratios-mean}In this figure, we plot $\ratioLin$ (\emph{left panel}) and $\ratioQuad$ (\emph{right panel}) as a function of the bandwidth in the change of mean scenario.  
    Note that both $\ratioLin$ and $\ratioQuad$ are becoming very ``flat'' when $\mu$ goes to~$0$, leading to numerical problems when maximizing with respect to~$\nu$.
}
\end{figure}

\paragraph{Change in the variance.}

In this case, the MMD is given by
\[
\expecunder{h(z,z')}{z,z'} = \nu\left(\frac{1}{\sqrt{\nu^2+2}}+\frac{1}{\sqrt{\nu^2+2\sigma^2}} - \frac{2}{\sqrt{\nu^2+\sigma^2+1}}\right)
\, .
\]
As in the previous paragraph, the second term in Eq.~\eqref{eq:lin-variance} and Eq.~\eqref{eq:quad-variance} is the MMD, which is given by the previous display.
As for the first term, we obtain for the linear-time statistic:
\begin{align*}
\expecunder{h(z,z')^2}{z,z'} &= \frac{\nu}{\sqrt{\nu^2+4}} + \frac{\nu}{\sqrt{\nu^2+4\sigma^2}} + \frac{2\nu}{\sqrt{\nu^2+2\sigma^2+2}} \\
&+ \frac{2\nu^2}{\sqrt{(\nu^2+2)(\nu^2+2\sigma^2)}} - \frac{4\nu^2}{\sqrt{(\nu^2+1)(\nu^2+\sigma^2)+2\nu^2+\sigma^2+1}} \\
&-\frac{4\nu^2}{\sqrt{(\nu^2+1)(\nu^2+\sigma^2)+\sigma^2(2\nu^2+\sigma^2+1)}} + \frac{2\nu^2}{\nu^2+\sigma^2+1}
\, .
\end{align*}
And for the the quadratic-time statistic:
\begin{align*}
\expecunder{\left(\expecunder{h(z,z')}{z'}\right)^2}{z} &= \nu^2 \times \biggl\{\frac{1}{\sqrt{(\nu^2+1)(\nu^2+3)}}+ \frac{1}{\sqrt{(\nu^2+\sigma^2)(\nu^2+3\sigma^2)}}\\
&+\frac{1}{\sqrt{(\nu^2+\sigma^2)(\nu^2+\sigma^2+2)}} + \frac{1}{\sqrt{(\nu^2+1)(\nu^2+2\sigma^2+1)}} \\
&+\frac{2}{\sqrt{(\nu^2+2\sigma^2)(\nu^2+2)}}-\frac{2}{\sqrt{(\nu^2+1)(\nu^2+\sigma^2)+2\nu^2+\sigma^2+1}} \\
&- \frac{2}{\sqrt{(\nu^2+2)(\nu^2+\sigma^2+1)}}-\frac{2}{\sqrt{(\nu^2+\sigma^2+1)(\nu^2+2\sigma^2)}} \\
&-\frac{2}{\sqrt{(\nu^2+1)(\nu^2+\sigma^2)+\sigma^2(2\nu^2+\sigma^2+1)}} + \frac{2}{\nu^2+\sigma^2+1}\biggr\}
\, .
\end{align*}

\begin{figure}[t!]
\centering 
\includegraphics[width=0.49\linewidth]{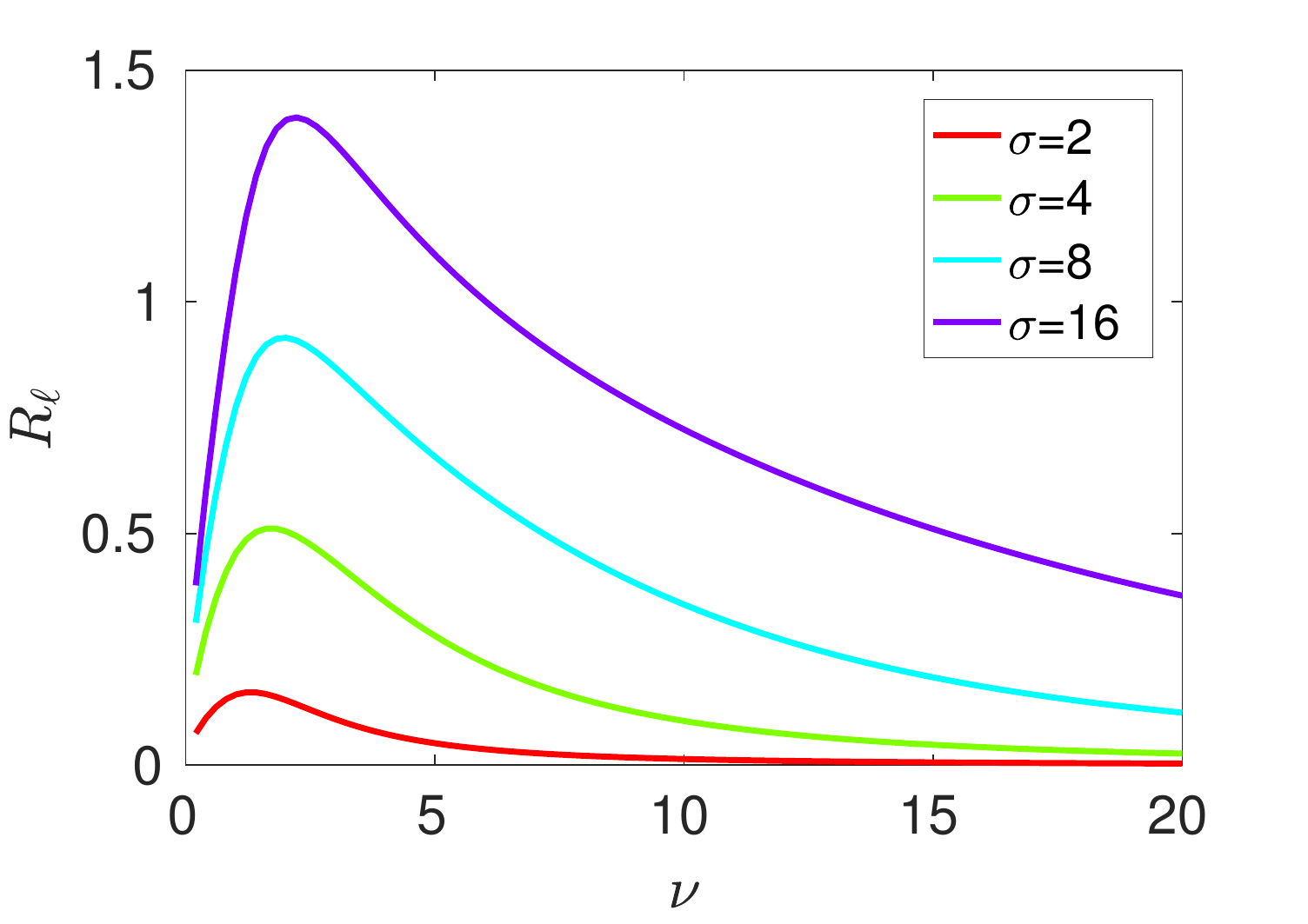}
\hspace{-0.5cm}
\includegraphics[width=0.49\linewidth]{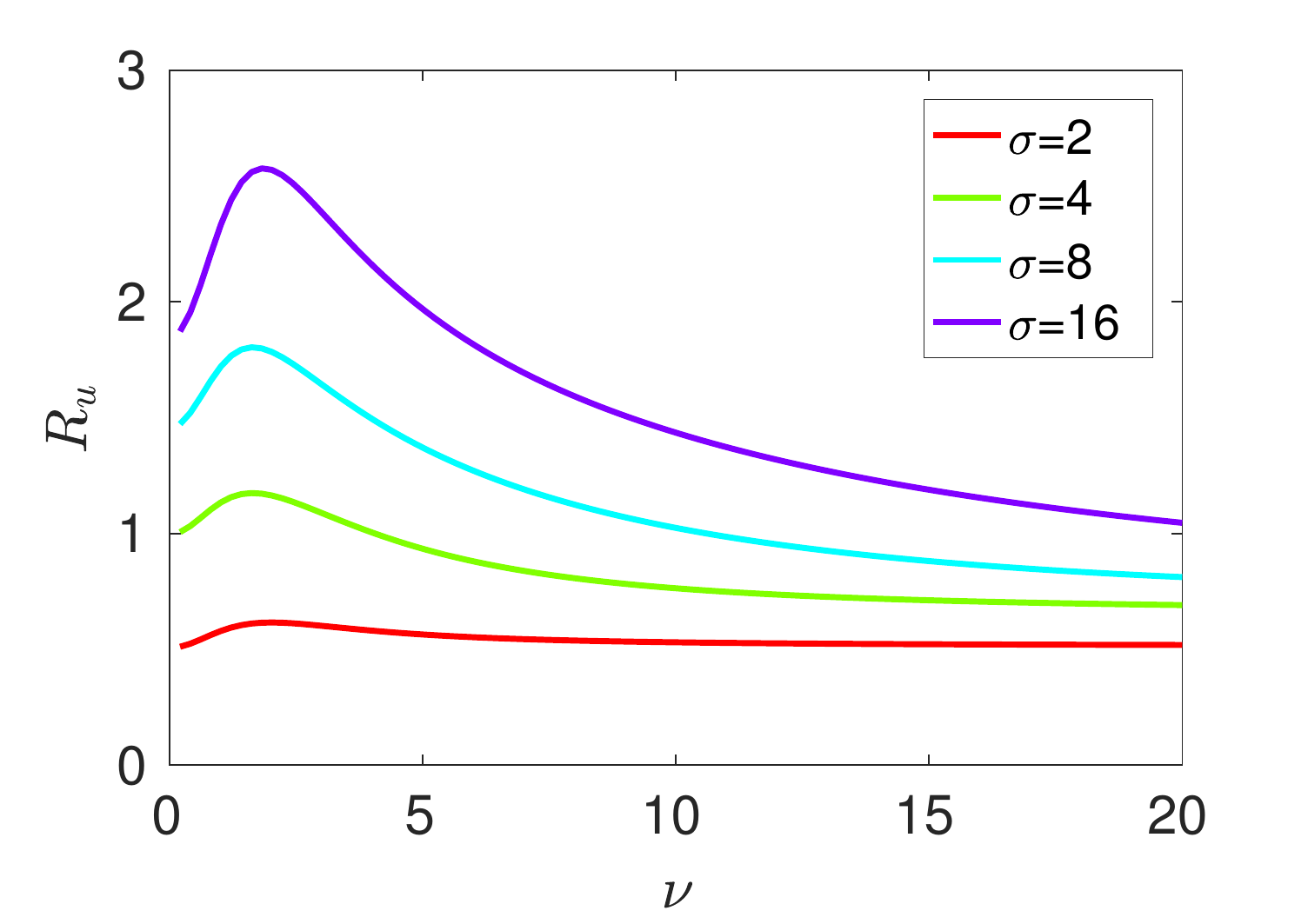}
\caption{\label{fig:ratios-var}In this figure, we plot $\ratioLin$ (\emph{left panel}) and $\ratioQuad$ (\emph{right panel}) as a function of the bandwidth in the change of variance scenario.  
    Note that both $\ratioLin$ and $\ratioQuad$ are becoming very ``flat'' when $\sigma$ goes to~$1$, leading to numerical problems when maximizing with respect to~$\nu$.
    This is similar to the phenomenon observed in the change in the mean scenario when $\mu\to 0$.
}
\end{figure}

\section{Proof of Prop.~\ref{prop:abs-computation}}
\label{sec:proof-abs-computation}

Set $T_n^{u}\defeq n\quadMMD^2$. 
According to Th.~12 in \citet{Gre_Bor_Ras:2012}, with our notations,
\[
n\quadMMD^2 \cvlaw \sum_{\ell=1}^{+\infty} \lambda_{\ell}\left[\left(\sqrt{2}a_{\ell}-\sqrt{2}b_{\ell}\right)^2-4\right]
\, ,
\]
where $a_{\ell},b_{\ell}$ are i.i.d. standard Gaussian random variables and the $\lambda_{\ell}$ are the eigenvalues of the centered Gaussian kernel operator.
We deduce that $n\quadMMD^2 \cvlaw 4\sum_{\ell=1}^{+\infty} \lambda_{\ell}(c_{\ell}^2-1)$, where the $c_{\ell}$ are i.i.d. standard Gaussian random variables. 
Denote by $F$ the cumulative distribution function of this random variables. 
According to \citet{Bah:1960}, the cumulative distribution function of $c_{\ell}^2$ belongs to $\Bclass{1}{1}$.
It follows from Th.~11 in \citet{Jit_Xu_Sza:2017} that $F$ belongs to $\Bclass{\frac{1}{4\lambda_1}}{1}$. 
We now take $R(n)=n$. 
According to Th.~10 in \citet{Gre_Bor_Ras:2012}, $T_n/R(n) = \quadMMD^2 \cvproba \MMD^2$. 
We conclude via Th.~9 in \citet{Jit_Xu_Sza:2017}.

Set $T_n^{\ell}\defeq \sqrt{n}\linMMD^2$. 
According to~\citet[Eq.~5]{Gre_Sej_Str:2012} the limiting distribution of $\sqrt{n}\linMMD$ under the null is $\gaussian{0}{8\sigma_{\ell}^2}$. 
The cumulative distribution function of a standard Gaussian random variable belongs to $\Bclass{1}{2}$ according to \citet{Bah:1960} and we deduce that the cumulative distribution function of $\gaussian{0}{8\sigma_{\ell}^2}$ belongs to $\Bclass{1/(8\sigma_{\ell}^2}{2}$. 
Set $R(n)\defeq \sqrt{n}$.
Since $T_n^{\ell}/R(n)=\linMMD^2$ converges in probability towards the squared MMD, we conclude again  with Th.~9 in \citet{Jit_Xu_Sza:2017}. 
\qed

\end{appendices}

\end{document}